\documentclass[amscd, amssymb,11pt]{article}

\usepackage[usenames]{color}
\definecolor{red}{rgb}{1.0,0.0,0.0}

\definecolor{blu}{rgb}{0.0,0.0,1.0}

\definecolor{gre}{rgb}{0.03,0.50,0.03}

\normalbaselines
\usepackage{amsmath}
\usepackage{amsthm}
\usepackage{amsfonts}
\usepackage{amssymb}
\usepackage{bbm}

\newtheorem{lemma}{Lemma}[section]

\newtheorem{theorem}{Theorem}[section]

\newtheorem{remark}{Remark}[section]
\newtheorem{definition}{Definition}[section]
\newtheorem{example}{Example}[section]

\let\Section=\section
\def\section{\setcounter{equation}{0}\Section}
\sf

\def\Swiech
{{\accent"13S}wie{\hbox{\kern -0.21em\lower
0.79ex\hbox{$\textfont1=\scriptfont1
\lhook$}}}ch}
\def\SWIECH
{{\accent"13S}WIE{\hbox{\kern -0.26em\lower
0.77ex\hbox{$\textfont1=\scriptfont1
\lhook$}}}CH}

\begin{document}

\title{{\bf Semiconcavity of viscosity solutions for a class of degenerate elliptic integro-differential equations in $\mathbb R^n$
}}

\author{
    \textsc{Chenchen Mou}\\
    \textit{School of Mathematics, Georgia Institute of Technology}\\
\textit{
Atlanta, GA 30332, U.S.A.}\\
 \textit{E-mail: cmou3@math.gatech.edu}    }
\date{}

\maketitle

\begin{abstract}
This paper is concerned with semiconcavity of viscosity solutions for a class of degenerate elliptic integro-differential equations in $\mathbb R^n$. This class of equations includes Bellman equations containing operators of L\'evy-It\^o type. H\"{o}lder and Lipschitz continuity of viscosity solutions for a more general class of degenerate elliptic integro-differential equations are also provided.
\end{abstract}

\vspace{.2cm}

\vspace{.2cm}
\section{Introduction}

In this paper, we study semiconcavity of viscosity solutions of integro-differential equations of the type 
\begin{equation}\label{eq1.1}
G(x,u,Du,D^2u,I[x,u])=0\quad\text{in}\,\,\mathbb R^n,
\end{equation}
where $\mathbb R^n$ is an $n$-dimensional Euclidean space and $I[x,u]$ is a L\'evy-It\^o operator. The function $u$ is real-valued and the L\'evy-It\^o operator $I$ has the form
\begin{equation*}
I[x,u]:=\int_{\mathbb R^n}[u(x+j(x,\xi))-u(x)-\mathbbm{1}_{B_1(0)}(\xi) Du(x)\cdot j(x,\xi)]\mu(d\xi),
\end{equation*}
where $\mathbbm{1}_{B_1(0)}$ denotes the indicator function of the unit ball ${B_1(0)}$, $j(x,\xi)$ is a function that determines the size of the jumps for the diffusion related to the operator $I$, and $\mu$ is a L\'evy measure. The L\'evy measure $\mu$ is a Borel measure on $\mathbb R^n\setminus\{0\}$ satisfying
\begin{equation}\label{eq:1.5}
\int_{\mathbb R^n\setminus\{0\}}\rho(\xi)^2\mu(d\xi)<+\infty,
\end{equation}
where $\rho:\mathbb R^n\setminus\{0\}\to\mathbb R^+$ is a Borel measurable, locally bounded function satisfying $\lim_{\xi\to 0}\rho(\xi)=0$ and $\inf_{\xi\in B_r^c(0)}\rho(\xi)>0$ for any $r>0$. We extend $\mu$ to a measure on $\mathbb R^n$ by setting $\mu(\{0\})=0$.  Our assumption on $\mu$ implies that $\mu(B_r^c(0))<+\infty$ for any $r>0$. The nonlinearity $G:\mathbb R^n\times \mathbb R\times \mathbb R^n\times \mathbb S^n\times\mathbb R\to \mathbb R$ is a continuous function which is coercive, i.e., there is a positive constant $\gamma$ such that, for any $x,p\in \mathbb R^n$, $r\geq s$, $X\in\mathbb S^n$, $l\in \mathbb R$,
\begin{equation}\label{eq1.3}
\gamma(r-s)\leq G(x,r,p,X,l)-G(x,s,p,X,l),
\end{equation}
and degenerate elliptic in a sense that, for 
any $x,p\in \mathbb R^n$, $r,l_1,l_2\in \mathbb R$, $X,Y\in\mathbb S^n$
\begin{equation}\label{eq1.2}
G(x,r,p,X,l_1)\leq G(x,r,p,Y,l_2)\quad\text{if}\,\, X\geq Y,\,l_1\geq l_2.
\end{equation}
Here $\mathbb S^n$ is the set of symmetric $n\times n$ matrices equipped with its usual order.
We will also be interested in equations of Bellman type
\begin{equation}\label{eq1.7}
\sup_{\alpha\in \mathcal A}\big\{-Tr\big(\sigma_{\alpha}(x)\sigma_{\alpha}^T(x)D^2u(x)\big)-I_{\alpha}[x,u]+b_{\alpha}(x)\cdot Du(x)+c_{\alpha}(x)u(x)+f_{\alpha}(x)\big\}=0,\quad \text{in}\,\,\mathbb R^n, 
\end{equation}
where $\sigma_{\alpha}:\mathbb R^n\to\mathbb R^{n\times m}$, $b_{\alpha}:\mathbb R^n\to \mathbb R^n$, $c_{\alpha}:\mathbb R^n\to\mathbb R$, $f_{\alpha}:\mathbb R^n\to\mathbb R$ are continuous functions,
\begin{equation*}
I_{\alpha}[x,u]=\int_{\mathbb R^n}[u(x+j_{\alpha}(x,\xi))-u(x)-\mathbbm{1}_{B_1(0)}(\xi) Du(x)\cdot j_{\alpha}(x,\xi)]\mu(d\xi)\,\,\text{and}\,\,c_{\alpha}\geq\gamma>0\,\,\text{in}\,\,\mathbb R^n.
\end{equation*}

The proof of semiconcavity of viscosity solutions is done in two steps. We first prove Lipschitz continuity of viscosity solutions. We then adapt to the nonlocal case the approach from \cite{IL,I} for obtaining semiconcavity of viscosity solutions of elliptic partial differential equations. In recent years, regularity theory of viscosity solutions of integro-differential equations has been studied by many authors under different types of ellipticity assumptions. It is impossible for us to make a complete review of all the related literature. However, the following are what we have in mind. Regularity results were initiated by assuming nondegenerate ellipticity of second order terms such as \cite{BL,GM,GL,MP1,MP2,MP3,MP4,MP5,MP6,MP7} for both elliptic and parabolic integro-differential equations. More recently, striking regularity results were obtained under uniform ellipticity assumption for nonlocal terms. This assumption, introduced by L. A. Caffarelli and L. Silvestre, is defined using nonlocal Pucci operators. Several H\"{o}lder, $C^{1,\alpha}$ and Shauder estimates for nonlocal fully nonlinear equations were obtained by various authors \cite{CS1,CS2,CS3,CD1,CD2,CD3,JX,Kri,Se1,Se,Si} under this uniform ellipticity assumption. The other notion of uniform ellipticity was defined by G. Barles, E. Chasseigne and C. Imbert. It requires either nondegeneracy of the nonlocal terms, or nondegeneracy of nonlocal terms in some directions and nondegeneracy of second order terms in the complementary directions. It was used to obtain H\"{o}lder and Lipschitz continuity for a class of mixed integro-differential equations, see \cite{BCCI,BCI}. 

In Section \ref{sec:holderlip}, we study H\"{o}lder and Lipschitz continuity of viscosity solutions for ($\ref{eq1.1}$) and equations of Bellman-Issacs type, i.e.,
\begin{equation}\label{eq..1.6}
\sup_{\alpha\in \mathcal A}\inf_{\beta\in \mathcal B}\big\{-Tr\big(\sigma_{\alpha\beta}(x)\sigma_{\alpha\beta}^T(x)D^2u(x)\big)-I_{\alpha\beta}[x,u]+b_{\alpha\beta}(x)\cdot Du(x)+c_{\alpha\beta}(x)u(x)+f_{\alpha\beta}(x)\big\}=0,\,\, \text{in}\,\,\mathbb R^n
\end{equation}
where $I_{\alpha\beta}[x,u]=\int_{\mathbb R^n}[u(x+j_{\alpha\beta}(x,\xi))-u(x)-\mathbbm{1}_{B_1(0)}(\xi) Du(x)\cdot j_{\alpha\beta}(x,\xi)]\mu(d\xi)$ and
\begin{equation}\label{eq.1.7}
c_{\alpha\beta}\geq\gamma>0\quad \text{in}\,\,\mathbb R^n.
\end{equation}
Our H\"{o}lder and Lipschitz continuity results are different from these of \cite{BCCI,BCI,Si} since we allow both the nonlocal terms and the second order terms to be degenerate. However, to compensate for degeneracy, we need to assume that the constant $\gamma$ appearing in ($\ref{eq1.3}$) and ($\ref{eq.1.7}$) is sufficiently large. The reader can consult \cite{jk2} for continuous dependence and continuity estimates for viscosity solutions of nonlinear degenerate parabolic integro-differential equations.

Having the Lipschitz continuity results, in Section \ref{sec:semi}  we derive the main results of this manuscript, i.e., semiconcavity of viscosity solutions of equations ($\ref{eq1.1}$) and ($\ref{eq1.7}$). To our knowledge, the only available results in this direction are about semiconcavity of viscosity solutions of time dependent integro-differential equations of Hamilton-Jacobi-Bellman (HJB) type whose proofs are based on probabilistic arguments. In \cite{J}, the author proved joint time-space semiconcavity of viscosity solutions of time dependent integro-differential equations of HJB type with terminal condition, using a representation formula based on forward and backward stochastic differential equations. However, the proof there depended on a restrictive assumption that the L\'evy measure $\mu$ is finite. In another paper \cite{F},  it was shown that the value function of an abstract infinite dimensional optimal control problem is $w$-semiconcave, if the data in the state evolution equation are $C^{1,w}$ and the data in the cost functional are $w$-semiconcave. The method was then applied to the finite dimensional Euclidean space providing semiconcavity result for the value function of a stochastic optimal control problem associated with a time dependent version of ($\ref{eq1.7}$). Later the auther extended the semiconcavity result in state variables to that in time and state variables jointly in \cite{F1}. Our result for ($\ref{eq1.7}$) extends results of \cite{F} to the time independent case and provide a different purely analytical approach. The result for ($\ref{eq1.1}$) is totally new since the solution may not have an explicit probabilistic representation formula and thus the analytical proof seems to be the only available method. Finally we remark that regarding semiconcavity of viscosity solutions of PDEs of HJB type, in addition to the already mentioned analytical proofs of \cite{IL,I}, other proofs by probabilistic methods can be found in \cite{FS,Kry,L1,L,LS,YZ}.

\section{Notation and Definitions}
\label{sec:notdef}
We will write $B_\delta(x)$ for the open ball centered at $x$ with radius $\delta>0$, $USC(\mathbb R^n)$ $(LSC(\mathbb R^n))$ for the space
of upper (lower) semi-continuous functions in $\mathbb R^n$ and $BUC(\mathbb R^n)$ for the space of bounded and uniformly continuous functions in $\mathbb R^n$. If $\Omega'$ is an open set, for each non-negative integer $k$ and $0<\alpha\leq1$, we denote by $C^{k,\alpha}(\Omega')$ ($C^{k,\alpha}(\bar{\Omega}')$) the subspace of $C^k(\Omega')$ ($C^{k}(\bar {\Omega}')$) consisting functions whose $k$th partial derivatives are locally (uniformly)  $\alpha$-H\"older continuous in $\Omega'$. We note that $C^{0,\alpha}(\mathbb R^n)$ ($C^{0,\alpha}(\bar{\mathbb R}^n)$) is the space of functions are locally (uniformly) $\alpha$-H\"older continuous in $\mathbb R^n$. We will use the standard notation
\[
\left[u\right]_{k,\alpha;\Omega'}:=\sup_{x,y\in\Omega',x\not=y,|j|=k}\frac{|\partial^ju(x)-\partial^ju(y)|}{|x-y|^\alpha},\quad \text{if}\,\,0<\alpha\leq1,
\]
 and 
\[
|u|_{k;\Omega'}:=\sup_{x\in\Omega',|j|=k}|\partial^ju(x)|,
\]
where $j=(j_1,j_2,\cdots,j_n)\in\mathbb N^n$, $|j|:=j_1+j_2+\cdots+j_n$ and $\partial^j u:=\frac{\partial^{|j|}u}{(\partial x_1)^{j_1}(\partial x_2)^{j_2}\cdots(\partial x_n)^{j_n}}$. For any $1<\theta'\leq 2$ and any convex open set $\Omega''$, we say a set of functions $\{f_\alpha\}_{\alpha\in\mathcal{A}}$ is uniformly $\theta'$-semiconvex with constant $C$ in $\Omega''$ if, for any $x,y\in\Omega''$, $\alpha\in\mathcal{A}$,
\begin{equation*}
2f_{\alpha}(\frac{x+y}{2})-f_\alpha(x)-f_\alpha(y)\leq C|x-y|^{\theta'}.
\end{equation*}
We say a set of functions $\{f_\alpha\}_{\alpha\in\mathcal{A}}$ is uniformly $\theta'$-semiconcave with constant $C$ in $\Omega''$ if $\{-f_\alpha\}_{\alpha\in\mathcal{A}}$ is uniformly $\theta'$-semiconvex with constant $C$ in $\Omega''$. If the set $\mathcal{A}$ is a unit set, i.e., $\mathcal{A}=\{\alpha_0\}$, then we just simply say that $f_{\alpha_0}$ is $\theta'$-semiconvex ($\theta'$-semiconcave) in $\Omega''$.

We then recall the definition of a 
viscosity solution 
of ($\ref{eq1.1}$). In order to do it, we introduce two associated operators $I^{1,\delta}$ and 
$I^{2,\delta}$,
\begin{equation*}
I^{1,\delta}[x,p,u]=\int_{|\xi|<\delta}[u(x+j(x,\xi))-u(x)-\mathbbm{1}_{B_1(0)}(\xi) p\cdot j(x,\xi)]\mu(d\xi), 
\end{equation*}
\begin{equation*}
I^{2,\delta}[x,p,u]=\int_{|\xi|\geq\delta}[u(x+j(x,\xi))-u(x)-\mathbbm{1}_{B_1(0)}(\xi) p\cdot j(x,\xi)]\mu(d\xi).
\end{equation*}
\begin{definition}\label{de2.1}
A bounded function $u\in USC(\mathbb R^n)$ is a viscosity subsolution of ($\ref{eq1.1}$) if whenever $u-\varphi$ has a maximum over 
$B_\delta(x)$ at $x\in \mathbb R^n$ for a test function $\varphi\in C^2(B_\delta(x))$, $\delta>0$, then
\begin{equation*}\label{eqq2.1}
G\big(x,u(x),D\varphi(x),D^2\varphi(x),I^{1,\delta}[x,D\varphi(x),\varphi]+I^{2,\delta}[x,D\varphi(x),u]\big)\leq 0.
\end{equation*}
A bounded function $u\in LSC(\mathbb R^n)$ is a viscosity supersolution of ($\ref{eq1.1}$) if whenever $u-\varphi$ has 
a minimum over $B_\delta(x)$ at $x\in \mathbb R^n$ for a test function $\varphi\in C^2(B_\delta(x))$, $\delta>0$, then
\begin{equation*}\label{eqq2.2}
G\big(x,u(x),D\varphi(x),D^2\varphi(x),I^{1,\delta}[x,D\varphi(x),\varphi]+I^{2,\delta}[x,D\varphi(x),u]\big)\geq 0.
\end{equation*}
A function $u$ is a viscosity solution of ($\ref{eq1.1}$) if it is both a viscosity subsolution and viscosity supersolution
of ($\ref{eq1.1}$).
\end{definition}

\section{H\"{o}lder and Lipschitz continuity}
\label{sec:holderlip}

In this section we prove the H\"{o}lder and Lipschitz continuity of viscosity solutions of ($\ref{eq1.1}$) and ($\ref{eq..1.6}$). We start with equation ($\ref{eq1.1}$). We make the following assumptions on the nonlinearity $G$ and the function $j(x,\xi)$.
\\
\\
($H1$) There are a constant $0<\theta\leq 1$, a non-negative constant $\Lambda$ and two positive constants $C_1,C_2$ such that, for any $x,y\in \mathbb R^n$, $r,l_x,l_y\in \mathbb R$, $X,Y\in\mathbb S^n$ and $L,\eta>0$, we have
\begin{eqnarray*}
&&G(y,r,L\theta|x-y|^{\theta-2}(x-y),Y,l_y)-G(x,r,L\theta|x-y|^{\theta-2}(x-y)+2\eta x,X,l_x)\\
&\leq&\Lambda(l_x-l_y)+C_1(1+L)|x-y|^\theta+C_2\eta(1+|x|^2),
\end{eqnarray*}
if
\begin{equation*}
\left(         
  \begin{array}{cc}   
    X  &   0\\  
    0   &-Y 
  \end{array}
\right)\leq L|x-y|^{\theta-2}
\left(         
  \begin{array}{cc}   
    I  &   -I\\  
    -I   & I 
  \end{array}
\right)+2\eta
\left(         
  \begin{array}{cc}   
    I   &   \,\,0\\  
 0 &\,\,   0 
  \end{array}
\right).
\end{equation*}
($H2$) For any $x,y\in \mathbb R^n$, we have
\begin{equation*}
|j(x,\xi)-j(y,\xi)|\leq |x-y|\rho(\xi)\quad\text{for}\,\,\xi\in \mathbb R^n,
\end{equation*}
\begin{equation*}
|j(0,\xi)|\leq \rho(\xi)\quad\text{for}\,\,\xi\in \mathbb R^n.
\end{equation*}
The following lemma is a nonlocal version of the Jensen-Ishii lemma we borrow from \cite{jk1}, Theorem $4.9$. The reader can consult \cite{BI} for a more general Jensen-Ishii lemma for integro-differential eqations, which allows for arbitrary growth of solutions at infinity. Before giving the lemma, we notice that our Definition \ref{de2.1} corresponds to the alternative definition of a viscosity solution in \cite{jk1}, see Lemma $4.8$. 
\begin{lemma}\label{le3.1}
Suppose that the nonlinearity $G$ in ($\ref{eq1.1}$) is continuous and satisfies ($\ref{eq:1.5}$)-($\ref{eq1.2}$). Let $u,v$ be bounded functions and be respectively a viscosity subsolution and a viscosity supersolution of 
\begin{equation*}
G(x,u,Du,D^2u,I[x,u])= 0\,\,\,\,\text{and}\,\,\,\,G(x,v,Dv,D^2v,I[x,v])= 0\quad\text{in}\,\,\mathbb R^n.
\end{equation*}
Let $\psi\in C^2(\mathbb R^{2n})$ and $(\hat x,\hat y)\in\mathbb R^n\times\mathbb R^n$ be such that 
\begin{equation*}
(x,y)\mapsto u(x)-v(y)-\psi(x,y)
\end{equation*}
has a global maximum at $(\hat x,\hat y)$. Furthermore, assume that in a neighborhood of $(\hat x,\hat y)$ there are continuous functions $g_0:\mathbb R^{2n}\to \mathbb R$, $g_1:\mathbb R^n\to\mathbb S^n$ with $g_0(\hat x,\hat y)>0$, satisfying 
\begin{equation*}
D^2\psi(x,y)\leq g_0(x,y)\left(         
  \begin{array}{cc}   
    I  &   -I\\  
    -I   & I 
  \end{array}
\right)+
\left(         
  \begin{array}{cc}   
    g_1(x)  &   0\\  
    0   & 0 
  \end{array}
\right).
\end{equation*}
Then, for any $0<\delta<1$ and $\epsilon_0>0$, there are $X,Y\in\mathbb S^n$ satisfying
\begin{equation*}
\left(         
  \begin{array}{cc}   
    X  &   0\\  
    0   &   -Y 
  \end{array}
\right)-\left(         
  \begin{array}{cc}   
    g_1(\hat x)  &   0\\  
    0   &   0 
  \end{array}
\right)
\leq (1+\epsilon_0)g_0(\hat x,\hat y)\left(         
  \begin{array}{cc}   
    I  &   -I\\  
    -I   & I 
  \end{array}
\right),
\end{equation*}
such that
\begin{equation*}
G\big(\hat x,u(\hat x),D_x\psi(\hat x,\hat y),X,I^{1,\delta}[\hat x,D_x\psi(\hat x,\hat y),\psi(\cdot,\hat y)]+I^{2,\delta}[\hat x,D_x\psi(\hat x,\hat y),u(\cdot)]\big)\leq 0,
\end{equation*}
\begin{equation*}
G\big(\hat y,v(\hat y),-D_y\psi(\hat x,\hat y),Y,I^{1,\delta}[\hat y,-D_y\psi(\hat x,\hat y),-\psi(\hat x,\cdot)]+I^{2,\delta}[\hat y,-D_y\psi(\hat x,\hat y),v(\cdot)]\big)\geq 0.
\end{equation*}
\end{lemma}
\begin{remark}\label{re3.1}
The statement of Lemma $\ref{le3.1}$ is weaker than Theorem $4.9$ in \cite{jk1}. By Theorem $4.9$ in \cite{jk1}, the same result as Lemma $\ref{le3.1}$ is also true for Bellman-Isaacs equations ($\ref{eq..1.6}$).
\end{remark}
\begin{lemma}\label{le..3.2}
Suppose that a L\'{e}vy measure $\mu$ satisfies ($\ref{eq:1.5}$) and $j(x,\xi)$ satisfies assumption ($H2$). Then we have 
\begin{eqnarray}
 M_1:&=&\sup_{x\not=y}\Big\{|x-y|^{-\theta}\int_{\mathbb R^n}\Big[|x-y+j(x,\xi)-j(y,\xi)|^\theta-|x-y|^\theta\nonumber\\
&&-\mathbbm{1}_{B_1(0)}(\xi)\theta|x-y|^{\theta-2}(x-y)\cdot\big(j(x,\xi)-j(y,\xi)\big)\Big]\mu(d\xi)\Big\}<+\infty.\label{eq,3.1}
\end{eqnarray}
\end{lemma}
\begin{proof}
We first define 
\begin{equation}\label{eq..3.2}
\phi(x,y)=|x-y|^\theta.
\end{equation}
By calculation, we have
\begin{equation}\label{eq:3.2}
D\phi(x,y)=\theta|x-y|^{\theta-2}\left(         
  \begin{array}{c}   
   x-y \\  
    y-x 
  \end{array}
\right),
\end{equation}
\begin{eqnarray}\label{eq:3.3}
D^2\phi(x,y)&=&\theta|x-y|^{\theta-2}\left(         
  \begin{array}{cc}   
    I  &   -I\\  
    -I   & I 
  \end{array}
\right)+
\theta(\theta-2)|x-y|^{\theta-4}\left(         
  \begin{array}{c}   
   x-y \\  
   y-x 
  \end{array}
\right)\otimes\left(         
  \begin{array}{c}   
   x-y \\  
    y-x 
  \end{array}
\right)\nonumber\\
&\leq&\theta|x-y|^{\theta-2}\left(         
  \begin{array}{cc}   
    I  &   -I\\  
    -I   & I 
  \end{array}
\right).
\end{eqnarray}
Since $\lim_{\xi\to 0}\rho(\xi)=0$, there exists a positive constant $\delta_1<1$ such that $\sup_{\xi\in B_{\delta_1}(0)}\rho(\xi)\leq \frac{1}{2}$. By ($\ref{eq:1.5}$), ($\ref{eq:3.2}$), ($\ref{eq:3.3}$) and ($H2$), we have, for any $x,y\in\mathbb R^n$ and $x\not=y$
\begin{eqnarray}
&&|x-y|^{-\theta}\int_{\mathbb R^n}\Big[|x-y+j(x,\xi)-j(y,\xi)|^\theta-|x-y|^\theta\nonumber\\
&&-\mathbbm{1}_{B_1(0)}(\xi)\theta|x-y|^{\theta-2}(x-y)\cdot\big(j(x,\xi)-j(y,\xi)\big)\Big]\mu(d\xi)\nonumber\\
&\leq&|x-y|^{-\theta}\theta \int_{B_{\delta_1}(0)}\Big(\sup_{0\leq t\leq 1}|x-y+t(j(x,\xi)-j(y,\xi))|^{\theta-2}|j(x,\xi)-j(y,\xi)|^2\Big)\mu(d\xi)\nonumber\\
&&+|x-y|^{-\theta}\int_{\mathbb R^n\setminus B_{\delta_1}(0)}\Big[|x-y+j(x,\xi)-j(y,\xi)|^{\theta}-|x-y|^\theta\nonumber\\
&&-\mathbbm{1}_{B_1(0)}(\xi)\theta|x-y|^{\theta-2}(x-y)\cdot\big(j(x,\xi)-j(y,\xi)\big)\Big]\mu(d\xi)\nonumber\\
&\leq&2^{2-\theta}\theta\int_{B_{\delta_1}(0)}\rho(\xi)^2\mu(d\xi)+\int_{\mathbb R^n\setminus B_{\delta_1}(0)}\rho(\xi)^\theta\mu(d\xi)+\theta\int_{B_1(0)\setminus B_{\delta_1}(0)}\rho(\xi)\mu(d\xi)<+\infty.\label{eq,,3.10}
\end{eqnarray}
\end{proof}
\begin{theorem}\label{th2.1}
Suppose that the nonlinearity $G$ in ($\ref{eq1.1}$) is continuous, and satisfies ($\ref{eq:1.5}$)-($\ref{eq1.2}$) and ($H1$). Suppose that $j(x,\xi)$ satisfies assumption ($H2$). Then, if $u\in BUC(\mathbb R^n)$ is a viscosity solution of ($\ref{eq1.1}$) and $\gamma>\Lambda M_1+C_1$ where $M_1$ is defined in ($\ref{eq,3.1}$), we have $u\in C^{0,\theta}(\bar{\mathbb R}^n)$.
\end{theorem}
\begin{proof} 
Let $\Phi(x,y)=u(x)-u(y)-\psi(x,y)$ where $\psi(x,y)=L\phi(x,y)+\eta|x|^2$ and $\phi(x,y)$ is defined in ($\ref{eq..3.2}$). We want to prove, for any $\eta>0$, we have $\Phi(x,y)\leq 0$ for all $x,y\in\mathbb R^n$ and some fixed sufficiently large $L$. Otherwise, there exists a positive constant $\eta_0$ such that $\sup_{x,y\in\mathbb R^n}\Phi(x,y)>0$ if $0<\eta<\eta_0$. By boundedness of $u$, there is a point ($\hat x,\hat y$) such that $\Phi(\hat x,\hat y)=\sup_{x,y\in\mathbb R^n}\Phi(x,y)>0$. Therefore, we have 
\begin{equation}\label{eq3.1}
\max\{\eta|\hat x|^2,L|\hat x-\hat y|^{\theta}\}< u(\hat x)-u(\hat y).
\end{equation} 
By ($\ref{eq:3.2}$) and ($\ref{eq:3.3}$), we obtain
\begin{equation*}
D^2\psi(\hat x,\hat y)\leq \theta L|\hat x-\hat y|^{\theta-2}\left(         
  \begin{array}{cc}   
    I  &   -I\\  
    -I   & I 
  \end{array}
\right)+2\eta\left(         
  \begin{array}{cc}   
    I   &  \,\,0\\  
    0 &   \,\,0 
  \end{array}
\right).
\end{equation*}
By Lemma $\ref{le3.1}$, since $u\in BUC(\mathbb R^n)$ is a viscosity solution of ($\ref{eq1.1}$), for any $0<\delta<1$ and $\epsilon_0>0$, there are $X,Y\in\mathbb S^n$ satisfying
\begin{equation}\label{eq:3.4}
\left(         
  \begin{array}{cc}   
    X  &   0\\  
    0   &   -Y 
  \end{array}
\right)-2\eta\left(         
  \begin{array}{cc}   
    I  &  \,\, 0\\  
    0   &   \,\,0 
  \end{array}
\right)
\leq (1+\epsilon_0)\theta L|\hat x-\hat y|^{\theta-2}\left(         
  \begin{array}{cc}   
    I  &   -I\\  
    -I   & I 
  \end{array}
\right),
\end{equation}
such that
\begin{equation*}
G\big(\hat x,u(\hat x),LD_x\phi(\hat x,\hat y)+2\eta\hat x,X,l_{\hat x}\big)\leq 0,
\end{equation*}
\begin{equation*}
G\big(\hat y,u(\hat y),-LD_y\phi(\hat x,\hat y),Y,l_{\hat y}\big)\geq 0,
\end{equation*}
where 
\begin{equation*}
l_{\hat x}=I^{1,\delta}[\hat x,LD_x\phi(\hat x,\hat y)+2\eta\hat x,L\phi(\cdot,\hat y)+\eta|\cdot|^2]+I^{2,\delta}[\hat x,LD_x\phi(\hat x,\hat y)+2\eta\hat x,u(\cdot)],
\end{equation*}
\begin{equation*}
l_{\hat y}=I^{1,\delta}[\hat y,-LD_y\phi(\hat x,\hat y),-L\phi(\hat x,\cdot)]+I^{2,\delta}[\hat y,-LD_y\phi(\hat x,\hat y),u(\cdot)].
\end{equation*}
Thus, by ($\ref{eq1.3}$), ($\ref{eq3.1}$) and ($H1$), we have
\begin{eqnarray}
\gamma L|\hat x-\hat y|^{\theta}&\leq&\gamma \big(u(\hat x)-u(\hat y)\big)\nonumber\\
&\leq&G\big(\hat y,u(\hat y),-LD_y\phi(\hat x,\hat y),Y,l_{\hat y}\big)-G\big(\hat x,u(\hat y),LD_x\phi(\hat x,\hat y)+2\eta\hat x,X,l_{\hat x}\big)\nonumber\\
&\leq&\Lambda(l_{\hat x}-l_{\hat y})+C_1(1+L)|\hat x-\hat y|^\theta+C_2\eta(1+|\hat x|^2).\label{eq::3.5}
\end{eqnarray}
Now we focus on estimating the integral term $l_{\hat x}-l_{\hat y}$. Thus,
\begin{eqnarray*}
l_{\hat x}-l_{\hat y}&=&L\int_{B_\delta(0)}\Big[|\hat x-\hat y+j(\hat x,\xi)|^\theta-|\hat x-\hat y|^\theta-\theta|\hat x-\hat y|^{\theta-2}(\hat x-\hat y)\cdot j(\hat x,\xi)\Big]\mu(d\xi)\\
&&+L\int_{B_\delta(0)}\Big[|\hat y-\hat x+j(\hat y,\xi)|^\theta-|\hat y-\hat x|^\theta-\theta|\hat y-\hat x|^{\theta-2}(\hat y-\hat x)\cdot j(\hat y,\xi)\Big]\mu(d\xi)\\
&&+\eta\int_{B_\delta(0)}\Big(|\hat x+j(\hat x,\xi)|^2-|\hat x|^2-2\hat x\cdot j(\hat x,\xi)\Big)\mu(d\xi)\\
&&+\int_{B_{\delta}^c(0)}\Big[u(\hat x+j(\hat x,\xi))-u(\hat x)-u(\hat y+j(\hat y,\xi))+u(\hat y)\\
&&-\mathbbm{1}_{B_1(0)}(\xi)\big(\theta L|\hat x-\hat y|^{\theta-2}(\hat x-\hat y)\big)\cdot\big(j(\hat x,\xi)-j(\hat y,\xi)\big)-\mathbbm{1}_{B_1(0)}(\xi)2\eta\hat x\cdot j(\hat x,\xi)\Big]\mu(d\xi).
\end{eqnarray*}
Since $\Phi(x,y)$ attains a global maximum at ($\hat x,\hat y$), we have, for any $\xi\in\mathbb R^n$,
\begin{eqnarray}\label{eq3.5}
&&u(\hat x+j(\hat x,\xi))-u(\hat x)-u(\hat y+j(\hat y,\xi))+u(\hat y)\nonumber\\
&\leq& L\Big(|\hat x-\hat y+j(\hat x,\xi)-j(\hat y,\xi)|^\theta-|\hat x-\hat y|^\theta\Big)+\eta\Big(|\hat x+j(\hat x,\xi)|^2-|\hat x|^2\Big).
\end{eqnarray}
Thus, by ($\ref{eq:3.3}$) and ($\ref{eq3.5}$), we have
\begin{eqnarray}\label{eq::::3.7}
l_{\hat x}-l_{\hat y}&\leq&\theta L\int_{B_{\delta}(0)}\Big(\sup_{0\leq t\leq1}|\hat x-\hat y+tj(\hat x,\xi)|^{\theta-2}|j(\hat x,\xi)|^2+\sup_{0\leq t\leq1}|\hat y-\hat x+tj(\hat y,\xi)|^{\theta-2}|j(\hat y,\xi)|^2\Big)\mu(d\xi)\nonumber\\
&&+\eta\int_{\mathbb R^n}\Big(|\hat x+j(\hat x,\xi)|^2-|\hat x|^2-\mathbbm{1}_{B_1(0)}(\xi)2\hat x\cdot j(\hat x,\xi)\Big)\mu(d\xi)\nonumber\\
&&+L\int_{B_\delta^c(0)}\Big[|\hat x-\hat y+j(\hat x,\xi)-j(\hat y,\xi)|^\theta-|\hat x-\hat y|^\theta\nonumber\\
&&-\mathbbm{1}_{B_1(0)}(\xi)\theta|\hat x-\hat y|^{\theta-2}(\hat x-\hat y)\cdot\big(j(\hat x,\xi)-j(\hat y,\xi)\big)\Big]\mu(d\xi).
\end{eqnarray}
We claim that $\eta|\hat x|^2\to 0$ as $\eta\to 0$. Since $u$ is bounded in $\mathbb R^n$, for any positive integer $k$, let $(x_k,y_k)$ be a point such that 
\begin{equation*}
u(x_k)-u(y_k)-L\phi(x_k,y_k)\geq M-\frac{1}{k},
\end{equation*}
where $M:=\sup_{x,y\in\mathbb R^n}\{u(x)-u(y)-L\phi(x,y)\}<+\infty$. Thus,
\begin{equation}\label{e3.11}
M-\frac{1}{k}-\eta|x_k|^2\leq \Phi(x_k,y_k)\leq\Phi(\hat x,\hat y)\leq M.
\end{equation}
Letting $\eta\to 0$ and then letting $k\to +\infty$ in ($\ref{e3.11}$), we have $\lim_{\eta\to 0}\Phi(\hat x,\hat y)=M$. If we notice that
\begin{equation*}
\Phi(\hat x,\hat y)+\eta|\hat x|^2=u(\hat x)-u(\hat y)-L\phi(\hat x,\hat y)\leq M,\quad\forall \eta>0,
\end{equation*} 
the claim follows. Since $u\in BUC(\mathbb R^n)$ and ($\ref{eq3.1}$) holds, we have 
\begin{equation*}
\epsilon_1\leq|\hat x-\hat y|\leq \epsilon_1^{-1},
\end{equation*}
where $\epsilon_1$ is a positive constant independent of $\eta$. Letting $\delta\rightarrow 0$ and then letting $\eta\rightarrow 0$ in ($\ref{eq::3.5}$), we have, by ($\ref{eq:1.5}$), ($\ref{eq::::3.7}$) and ($H2$),
\begin{eqnarray*}
\gamma L|\hat x-\hat y|^{\theta}&\leq&\Lambda L\int_{\mathbb R^n}\Big[|\hat x-\hat y+j(\hat x,\xi)-j(\hat y,\xi)|^\theta-|\hat x-\hat y|^\theta\\
&&-\mathbbm{1}_{B_1(0)}(\xi)\theta|\hat x-\hat y|^{\theta-2}(\hat x-\hat y)\cdot\big(j(\hat x,\xi)-j(\hat y,\xi)\big)\Big]\mu(d\xi)+C_1(1+L)|\hat x-\hat y|^\theta.
\end{eqnarray*}
Therefore, by Lemma $\ref{le..3.2}$,
\begin{eqnarray}\label{eq,,3.9}
\gamma&\leq&\Lambda|\hat x-\hat y|^{-\theta}\int_{\mathbb R^n}\Big[|\hat x-\hat y+j(\hat x,
\xi)-j(\hat y,\xi)|^\theta-|\hat x-\hat y|^\theta\nonumber\\
&&-\mathbbm{1}_{B_1(0)}(\xi)\theta|\hat x-\hat y|^{\theta-2}(\hat x-\hat y)\cdot\big(j(\hat x,\xi)-j(\hat y,\xi)\big)\Big]\mu(d\xi)+C_1(1+\frac{1}{L})\nonumber\\
&\leq&\Lambda M_1+C_1(1+\frac{1}{L})<+\infty,
\end{eqnarray}
where $M_1$ is defined in ($\ref{eq,3.1}$). It is now obvious from ($\ref{eq,,3.9}$) that, if $\gamma>\Lambda M_1+C_1$, we can find a sufficiently large $L$ such that we have a contradiction. Therefore, we have $u\in C^{0,\theta}(\bar{\mathbb R}^n)$.
\end{proof}

Let us consider another important fully nonlinear integro-PDE appearing in the study of stochastic optimal control and stochastic differential games for processes with jumps, namely the Bellman-Isaacs equation ($\ref{eq..1.6}$). Equation ($\ref{eq..1.6}$) is not of the same form as ($\ref{eq1.1}$), which means that the following theorem is not a corollary of Theorem $\ref{th2.1}$.
\begin{theorem}\label{th2.2}
Suppose that $c_{\alpha\beta}\geq\gamma$ in $\mathbb R^n$ uniformly in $\alpha\in\mathcal{A},\beta\in\mathcal{B}$. Suppose that the L\'evy measure $\mu$ satisfies ($\ref{eq:1.5}$), and the family $\{j_{\alpha\beta}(x,\xi)\}$ satisfies assumption ($H2$) uniformly in $\alpha\in \mathcal A,\beta\in\mathcal B$. Suppose moreover that there exist a positive constant $C$ and $0<\theta\leq 1$ such that 
\begin{equation}\label{eq::3.8}
\sup_{\alpha\in\mathcal{A},\beta\in\mathcal{B}}\max\{|\sigma_{\alpha\beta}(0)|,|b_{\alpha\beta}(0)|\}<C,
\end{equation}
and
\begin{equation}\label{eq::3.9}
\sup_{\alpha\in\mathcal{A},\beta\in\mathcal{B}}\max\{[\sigma_{\alpha\beta}]_{0,1;\mathbb R^n},[b_{\alpha\beta}]_{0,1;\mathbb R^n},[c_{\alpha\beta}]_{0,\theta,\mathbb R^n},[f_{\alpha\beta}]_{0,\theta,\mathbb R^n}\}<+\infty.
\end{equation}
Then, if $u\in BUC(\mathbb R^n)$ is a viscosity solution of ($\ref{eq..1.6}$) and $\gamma>N_1$ where 
\begin{eqnarray}
N_1:&=&\sup_{x\not=y}\sup_{\alpha\in\mathcal{A},\beta\in\mathcal{B}}\Big\{\theta |x-y|^{-2}Tr\Big[\big(\sigma_{\alpha\beta}(x)-\sigma_{\alpha\beta}(y)\big)(\sigma_{\alpha\beta}(x)-\sigma_{\alpha\beta}(y)\big)^T\Big]\nonumber\\
&&+\theta |x-y|^{-2}\Big(b_{\alpha\beta}(y)-b_{\alpha\beta}(x)\Big)\cdot(x-y)+|x-y|^{-\theta}\int_{\mathbb R^n}\Big[|x-y+j_{\alpha\beta}(x,\xi)-j_{\alpha\beta}(y,\xi)|^\theta\nonumber\\
&&-|x-y|^\theta-\mathbbm{1}_{B_1(0)}(\xi)\theta|x-y|^{\theta-2}(x-y)\cdot\big(j_{\alpha\beta}(x,\xi)-j_{\alpha\beta}(y,\xi)\big)\Big]\mu(d\xi)\Big\}<+\infty,\label{eq,3.12}
\end{eqnarray}
we have $u\in C^{0,\theta}(\bar{\mathbb R}^n)$.
\end{theorem}
\begin{proof}
At the beginning of the proof, we will show that the  constant $N_1$ has an upper bound. By ($\ref{eq::3.9}$) and the estimates in ($\ref{eq,,3.10}$), we have
\begin{eqnarray*}
N_1&\leq&\theta\sup_{\alpha\in\mathcal{A},\beta\in\mathcal{B}}[\sigma_{\alpha\beta}]_{0,1;\mathbb R^n}^2+\theta\sup_{\alpha\in\mathcal{A},\beta\in\mathcal{B}}[b_{\alpha\beta}]_{0,1;\mathbb R^n}+2^{2-\theta}\theta\int_{B_{\delta_1}(0)}\rho(\xi)^2\mu(d\xi)\\
&&+\int_{\mathbb R^n\setminus B_{\delta_1}(0)}\rho(\xi)^\theta\mu(d\xi)+\theta\int_{B_1(0)\setminus B_{\delta_1}(0)}\rho(\xi)\mu(d\xi)<+\infty,
\end{eqnarray*}
where $\delta_1$ was chosen in Lemma $\ref{le..3.2}$.

Then we want to prove that, for any $\eta>0$, we have $\Phi(x,y)=u(x)-u(y)-\psi(x,y)\leq 0$ for all $x,y\in\mathbb R^n$ and some fixed sufficiently large $L$ where $\psi(x,y)$ is given in Theorem $\ref{th2.1}$. Otherwise, there exists a positive constant $\eta_0$ such that $\sup_{x,y\in\mathbb R^n}\Phi(x,y)>0$ if $0<\eta<\eta_0$. By boundedness of $u$, there is a point ($\hat x,\hat y$) such that $\Phi(\hat x,\hat y)=\sup_{x,y\in\mathbb R^n}\Phi(x,y)>0$. Therefore, we have ($\ref{eq3.1}$). By Remark $\ref{re3.1}$, since $u\in BUC(\mathbb R^n)$ is a viscosity solution of ($\ref{eq..1.6}$), for any $0<\delta<1$ and $\epsilon_0>0$, there are $X,Y\in\mathbb S^n$ satisfying ($\ref{eq:3.4}$) such that
\begin{equation*}
\sup_{\alpha\in \mathcal A}\inf_{\beta\in \mathcal{B}}\big\{-Tr\big(\sigma_{\alpha\beta}(\hat x)\sigma_{\alpha\beta}^T(\hat x)X\big)-l_{\hat x,\alpha\beta}+b_{\alpha\beta}(\hat x)\cdot D_x\psi(\hat x,\hat y)+c_{\alpha\beta}(\hat x)u(\hat x)+f_{\alpha\beta}(\hat x)\big\}\leq0,
\end{equation*}
\begin{equation*}
\sup_{\alpha\in \mathcal A}\inf_{\beta\in \mathcal{B}}\big\{-Tr\big(\sigma_{\alpha\beta}(\hat y)\sigma_{\alpha\beta}^T(\hat y)Y\big)-l_{\hat y,\alpha\beta}-b_{\alpha\beta}(\hat y)\cdot D_y\psi(\hat x,\hat y)+c_{\alpha\beta}(\hat y)u(\hat y)+f_{\alpha\beta}(\hat y)\big\}\geq0,
\end{equation*}
where 
\begin{equation*}
l_{\hat x,\alpha\beta}=I_{\alpha\beta}^{1,\delta}[\hat x,D_x\psi(\hat x,\hat y),\psi(\cdot,\hat y)]+I_{\alpha\beta}^{2,\delta}[\hat x,D_x\psi(\hat x,\hat y),u(\cdot)],
\end{equation*}
\begin{equation*}
l_{\hat y,\alpha\beta}=I_{\alpha\beta}^{1,\delta}[\hat y,-D_y\psi(\hat x,\hat y),-\psi(\hat x,\cdot)]+I_{\alpha\beta}^{2,\delta}[\hat y,-D_y\psi(\hat x,\hat y),u(\cdot)].
\end{equation*}
Since ($\ref{eq:3.2}$) and ($\ref{eq3.1}$) hold, and $c_{\alpha\beta}\geq \gamma$ in $\mathbb R^n$ uniformly in $\alpha\in\mathcal{A},\beta\in\mathcal{B}$, we have
\begin{equation}\label{eq3.7}
\gamma L|\hat x-\hat y|^\theta\leq\sup_{\alpha\in\mathcal{A},\beta\in\mathcal{B}}\Big\{L_{\alpha\beta}+N_{\alpha\beta}\Big\},
\end{equation}
where
\begin{eqnarray*}
L_{\alpha\beta}&=&Tr\Big(\sigma_{\alpha\beta}(\hat x)\sigma_{\alpha\beta}^T(\hat x)X-\sigma_{\alpha\beta}(\hat y)\sigma_{\alpha\beta}^T(\hat y)Y\Big)+\Big(b_{\alpha\beta}(\hat y)-b_{\alpha\beta}(\hat x)\Big)\cdot LD_x\phi(\hat x,\hat y)\\
&&+\Big(c_{\alpha\beta}(\hat y)-c_{\alpha\beta}(\hat x)\Big)u(\hat y)+f_{\alpha\beta}(\hat y)-f_{\alpha\beta}(\hat x)-2\eta b_{\alpha\beta}(\hat x)\cdot\hat x,
\end{eqnarray*}
and
\begin{equation*}
N_{\alpha\beta}=l_{\hat x,\alpha\beta}-l_{\hat y,\alpha\beta}.
\end{equation*}
By ($\ref{eq:3.4}$), ($\ref{eq::3.8}$) and ($\ref{eq::3.9}$), we see that (see also Example 3.6 in \cite{MHP})
\begin{eqnarray*}
&&Tr\Big(\sigma_{\alpha\beta}(\hat x)\sigma_{\alpha\beta}^T(\hat x)X-\sigma_{\alpha\beta}(\hat y)\sigma_{\alpha\beta}^T(\hat y)Y\Big)\\
&\leq& (1+\epsilon_0)\theta L|\hat x-\hat y|^{\theta-2}Tr\Big[\big(\sigma_{\alpha\beta}(\hat x)-\sigma_{\alpha\beta}(\hat y)\big)(\sigma_{\alpha\beta}(\hat x)-\sigma_{\alpha\beta}(\hat y)\big)^T\Big]+2\eta Tr\big(\sigma_{\alpha\beta}(\hat x)\sigma_{\alpha\beta}^T(\hat x)\big)\\
&\leq& (1+\epsilon_0)\theta L|\hat x-\hat y|^{\theta-2}Tr\Big[\big(\sigma_{\alpha\beta}(\hat x)-\sigma_{\alpha\beta}(\hat y)\big)(\sigma_{\alpha\beta}(\hat x)-\sigma_{\alpha\beta}(\hat y)\big)^T\Big]\\
&&+2\eta(C+\sup_{\alpha\in\mathcal{A},\beta\in\mathcal B}[\sigma_{\alpha\beta}]_{0,1;\mathbb R^n}|\hat x|)^2.
\end{eqnarray*}
Thus, we can estimate the local term $L_{\alpha\beta}$ easily. Using ($\ref{eq:3.2}$), ($\ref{eq::3.8}$), ($\ref{eq::3.9}$) and boundedness of $u$, we obtain
\begin{eqnarray}\label{eq::::3.12}
L_{\alpha\beta}&\leq&(1+\epsilon_0)\theta L|\hat x-\hat y|^{\theta-2}Tr\Big[\big(\sigma_{\alpha\beta}(\hat x)-\sigma_{\alpha\beta}(\hat y)\big)(\sigma_{\alpha\beta}(\hat x)-\sigma_{\alpha\beta}(\hat y)\big)^T\Big]\nonumber\\
&&+2\eta(C+\sup_{\alpha\in\mathcal{A},\beta\in\mathcal B}[\sigma_{\alpha\beta}]_{0,1;\mathbb R^n}|\hat x|)^2+\theta L|\hat x-\hat y|^{\theta-2}\Big(b_{\alpha\beta}(\hat y)-b_{\alpha\beta}(\hat x)\Big)\cdot (\hat x-\hat y)\nonumber\\
&&+\sup_{\alpha\in\mathcal{A},\beta\in\mathcal{B}}[c_{\alpha\beta}]_{0,\theta;\mathbb R^n}|u|_{0;\mathbb R^n}|\hat x-\hat y|^{\theta}+\sup_{\alpha\in\mathcal{A},\beta\in\mathcal{B}}[f_{\alpha\beta}]_{0,\theta;\mathbb R^n}|\hat x-\hat y|^{\theta}\nonumber\\
&&+2\eta(C|\hat x|+\sup_{\alpha\in\mathcal{A},\beta\in\mathcal B}[b_{\alpha\beta}]_{0,1;\mathbb R^n}|\hat x|^2).
\end{eqnarray}
Similarly as in the proof of Theorem $\ref{th2.1}$, we have $\eta|\hat x|^2\to 0$ as $\eta\to 0$ and
\begin{equation*}
\epsilon_1\leq|\hat x-\hat y|\leq \epsilon_1^{-1},
\end{equation*}
where $\epsilon_1$ is a positive constant independent of $\eta$.
Letting $\delta\to 0$, $\eta\to 0$ and $\epsilon_0\to 0$ in ($\ref{eq3.7}$), we have, by ($\ref{eq::::3.12}$) and the same estimates on the nonlocal term $N_{\alpha\beta}$ as Theorem $\ref{th2.1}$,
\begin{eqnarray*}
\gamma L|\hat x-\hat y|^\theta&\leq&\sup_{\alpha\in\mathcal{A},\beta\in\mathcal{B}}L\Big\{\theta|\hat x-\hat y|^{\theta-2}Tr\Big[\big(\sigma_{\alpha\beta}(\hat x)-\sigma_{\alpha\beta}(\hat y)\big)(\sigma_{\alpha\beta}(\hat x)-\sigma_{\alpha\beta}(\hat y)\big)^T\Big]\\
&&+\theta |\hat x-\hat y|^{\theta-2}\Big(b_{\alpha\beta}(\hat y)-b_{\alpha\beta}(\hat x)\Big)\cdot (\hat x-\hat y)+\int_{\mathbb R^n}\Big[|\hat x-\hat y+j_{\alpha\beta}(\hat x,\xi)-j_{\alpha\beta}(\hat y,\xi)|^\theta\\
&&-|\hat x-\hat y|^\theta-\mathbbm{1}_{B_1(0)}(\xi)\theta|\hat x-\hat y|^{\theta-2}(\hat x-\hat y)\cdot\big(j_{\alpha\beta}(\hat x,\xi)-j_{\alpha\beta}(\hat y,\xi)\big)\Big]\mu(d\xi)\Big\}\\
&&+\sup_{\alpha\in\mathcal{A},\beta\in\mathcal{B}}[c_{\alpha\beta}]_{0,\theta;\mathbb R^n}|u|_{0;\mathbb R^n}|\hat x-\hat y|^{\theta}+\sup_{\alpha\in\mathcal{A},\beta\in\mathcal{B}}[f_{\alpha\beta}]_{0,\theta;\mathbb R^n}|\hat x-\hat y|^{\theta}.
\end{eqnarray*}
Therefore,
\begin{eqnarray}\label{eq,,3.16}
\gamma&\leq&\sup_{\alpha\in\mathcal{A},\beta\in\mathcal{B}}\Big\{\theta|\hat x-\hat y|^{-2}Tr\Big[\big(\sigma_{\alpha\beta}(\hat x)-\sigma_{\alpha\beta}(\hat y)\big)(\sigma_{\alpha\beta}(\hat x)-\sigma_{\alpha\beta}(\hat y)\big)^T\Big]\nonumber\\
&&+\theta |\hat x-\hat y|^{-2}\Big(b_{\alpha\beta}(\hat y)-b_{\alpha\beta}(\hat x)\Big)\cdot(\hat x-\hat y)+|\hat x-\hat y|^{-\theta}\int_{\mathbb R^n}\Big[|\hat x-\hat y+j_{\alpha\beta}(\hat x,\xi)-j_{\alpha\beta}(\hat y,\xi)|^\theta\nonumber\\
&&-|\hat x-\hat y|^\theta-\mathbbm{1}_{B_1(0)}(\xi)\theta|\hat x-\hat y|^{\theta-2}(\hat x-\hat y)\cdot\big(j_{\alpha\beta}(\hat x,\xi)-j_{\alpha\beta}(\hat y,\xi)\big)\Big]\mu(d\xi)\Big\}\nonumber\\
&&+\frac{1}{L}\sup_{\alpha\in\mathcal{A},\beta\in\mathcal{B}}[c_{\alpha\beta}]_{0,\theta;\mathbb R^n}|u|_{0;\mathbb R^n}+\frac{1}{L}\sup_{\alpha\in\mathcal{A},\beta\in\mathcal{B}}[f_{\alpha\beta}]_{0,\theta;\mathbb R^n}\nonumber\\
&\leq&N_1+\frac{1}{L}\sup_{\alpha\in\mathcal{A},\beta\in\mathcal{B}}[c_{\alpha\beta}]_{0,\theta;\mathbb R^n}|u|_{0;\mathbb R^n}+\frac{1}{L}\sup_{\alpha\in\mathcal{A},\beta\in\mathcal{B}}[f_{\alpha\beta}]_{0,\theta;\mathbb R^n},
\end{eqnarray}
where $N_1$ is defined in ($\ref{eq,3.12}$). It now follows from ($\ref{eq,,3.16}$) that, if $\gamma>N_1$, we can find a sufficiently large $L$ such that we have a contradiction. Therefore, we have $u\in C^{0,\theta}(\bar{\mathbb R}^n)$.
\end{proof}

\section{Semiconcavity}
\label{sec:semi}

In this section we investigate the semiconcavity of viscosity solutions of ($\ref{eq1.1}$) and ($\ref{eq1.7}$). Again we start with equation ($\ref{eq1.1}$). We impose the following conditions on $G$ and $j(x,\xi)$.
\\
\\
($\bar{H}1$) If $\varphi\in C^{0,1}(\bar{\mathbb R}^n)$, there are a constant $1<\bar\theta\leq2$, a non-negative constant $\Lambda$ and two positive constants $C_3,C_4$ such that, for any $x,y,z\in\mathbb R^n$, $l_x,l_y,l_z\in\mathbb R$, $X,Y,Z\in\mathbb S^n$ and $L,\eta>0$, we have
\begin{eqnarray}\label{e4.1}
&&2G(z,\varphi(z),-\frac{L}{2}D_z\phi(x,y,z),\frac{Z}{2},l_z)\nonumber\\
&&-G(x,\varphi(x),LD_x\phi(x,y,z)+2\eta x,X,l_x)-G(y,\varphi(y),LD_y\phi(x,y,z),Y,l_y)\nonumber\\
&\leq&-\gamma\big(\varphi(x)+\varphi(y)-2\varphi(z)\big)+\Lambda(l_x+l_y-2l_z)+C_3(1+L)\phi(x,y,z)+C_4\eta(1+|x|^2),
\end{eqnarray}
if 
\begin{eqnarray}\label{e4.2}
\left(         
  \begin{array}{ccc}   
    X & 0 & 0  \\  
    0& Y & 0\\
   0 & 0 & -Z
  \end{array}
\right)
&\leq&\frac{L}{\phi(x,y,z)}\left[\bar\theta(2\bar\theta-1)|x-y|^{2\bar{\theta}-2}
\left(         
  \begin{array}{ccc}   
    I & -I &\,\, 0  \\  
    -I& I &\,\, 0\\
   0 & 0 & \,\,0
  \end{array}
\right)+
\left(         
  \begin{array}{ccc}   
    I & I & -2I  \\  
    I& I &   -2I\\
   -2I & -2I &  4I
  \end{array}
\right)\right]\nonumber\\
&&+2\eta\left(         
  \begin{array}{ccc}   
    I &\,\, 0 &\,\, 0  \\  
    0&\,\, 0 &\,\, 0\\
   0 &\,\, 0 &\,\,0
  \end{array}
\right),
\end{eqnarray}
where $\gamma$ is given by ($\ref{eq1.3}$) and $\phi(x,y,z)=(|x-y|^{2\bar\theta}+|x+y-2z|^2)^{\frac{1}{2}}$.\\
($\bar{H}2$) ($H2$) holds and, with the same $\bar\theta$ in ($\bar{H}1$) and for any $x,y\in\mathbb R^n$, we have
\begin{equation*}
|j(x,\xi)+j(y,\xi)-2j(\frac{x+y}{2},\xi)|\leq |x-y|^{\bar\theta}\rho(\xi)\quad\text{for}\,\,\xi\in \mathbb R^n.
\end{equation*}
\begin{example}
Since the assumption ($\bar{H}1$) is complicated, we provide a concrete example to show when it is satisfied. We consider the nonlinear convex nonlocal equation
\begin{equation}\label{eq::4.1}
-Tr\big(\sigma(x)\sigma^T(x)D^2u(x)\big)+F(I[x,u])+b(x)\cdot Du(x)+c(x)u(x)+f(x)=0,\quad \text{in}\,\,\mathbb R^n,
\end{equation}
where $F:\mathbb R\to\mathbb R$ is a continuous function. Suppose the following conditions are satisfied: there exists a non-negative constant $\Lambda$ such that, for any $l_x,l_y\in\mathbb R$,
\begin{equation*}
c\geq \gamma\,\,\text{in}\,\,\mathbb R^n\,\,\text{and}\,\,c\in C^{1,\bar\theta-1}(\bar{\mathbb R}^n),
\end{equation*} 
\begin{equation*}
f\,\,\text{is}\,\,\bar\theta\text{-semiconvex}\,\,\text{in}\,\,\mathbb R^n\,\,\text{and}\,\,\max\{[\sigma]_{0,1;\mathbb R^n},[\sigma]_{1,\bar\theta-1;\mathbb R^n},[b]_{0,1;\mathbb R^n},[b]_{1,\bar\theta-1,\mathbb R^n},[f]_{0,1,\mathbb R^n}\}<+\infty,
\end{equation*}
\begin{equation}\label{eq::4.2}
F\,\,\text{is}\,\,\text{convex}\,\,\text{in}\,\,\mathbb R^n\,\,\text{and}\,\,F(l_y)-F(l_x)\leq\Lambda(l_x-l_y).
\end{equation}
By the estimates on the local terms in Theorem $\ref{th4.2}$, if equation ($\ref{eq::4.1}$) does not contain the nonlocal term $F(I[x,u])$, then ($\ref{eq::4.1}$) satisfies ($\bar{H}1$). Thus, we only need to estimate the nonlocal terms. For any $l_x,l_y,l_z$, we have, by ($\ref{eq::4.2}$),
\begin{eqnarray*}
2F(l_z)-F(l_x)-F(l_y)&\leq&2F(l_z)-2F(\frac{l_x+l_y}{2})+\Big(2F(\frac{l_x+l_y}{2})-F(l_x)-F(l_y)\Big)\\
&\leq&\Lambda(l_x+l_y-2l_z).
\end{eqnarray*}
Therefore, equation ($\ref{eq::4.1}$) satisfies ($\bar{H}1$). 

This example can be generalized to equation
\begin{equation}\label{e4.5}
G(x,u,Du,D^2u)+F(I[x,u])=0,\quad \text{in}\,\,\mathbb R^n,
\end{equation}
where $G$ satisfies ($\ref{e4.1}$) without the last argument if $\varphi\in C^{0,1}(\bar{\mathbb R}^n)$ and ($\ref{e4.2}$) holds, and $F$ satisfies ($\ref{eq::4.2}$). It is obvious that ($\bar{H}1$) holds for equation ($\ref{e4.5}$).
\end{example}
\begin{lemma}\label{le4.1}
Suppose that the nonlinearity $G$ in ($\ref{eq1.1}$) is continuous and satisfies ($\ref{eq:1.5}$)-($\ref{eq1.2}$). Let $u,v,w$ be bounded functions and be respectively a viscosity subsolution, a viscosity subsolution and a viscosity supersolution of 
\begin{equation*}
G(x,u,Du,D^2u,I[x,u])=0,\quad\text{in}\,\,\mathbb R^n,
\end{equation*}
\begin{equation*}
G(x,v,Dv,D^2v,I[x,v])=0,\quad\text{in}\,\,\mathbb R^n,
\end{equation*}
\begin{equation*}
G(x,w,Dw,D^2w,I[x,w])=0,\quad\text{in}\,\,\mathbb R^n.
\end{equation*}
Let $\psi\in C^2(\mathbb R^{3n})$ and $(\hat x,\hat y,\hat z)\in\mathbb R^n\times\mathbb R^n\times\mathbb R^n$ be such that 
\begin{equation*}
(x,y,z)\mapsto u(x)+v(y)-2w(z)-\psi(x,y,z)
\end{equation*}
has a global maximum at $(\hat x,\hat y,\hat z)$. Furthermore, assume that in a neighborhood of $(\hat x,\hat y,\hat z)$ there are continuous functions $g_0,g_1:\mathbb R^{3n}\to \mathbb R$, $g_2:\mathbb R^{n}\to\mathbb S^n$ with $g_1(\hat x,\hat y,\hat z)>0$, satisfying 
\begin{equation*}
D^2\psi(x,y,z)\leq g_0(x,y,z)\left(         
  \begin{array}{ccc}   
    I  &   -I  &\,\, 0\\  
    -I   & I   & \,\, 0 \\
   0 & 0 & \,\,0
  \end{array}
\right)+g_1(x,y,z)
\left(         
  \begin{array}{ccc}   
    I  &   I & -2I\\  
    I   & I  &  -2I\\
   -2I  &-2I & 4I 
  \end{array}
\right)+\left(         
  \begin{array}{ccc}   
    g_2(x)  & \,\,  0  &\,\, 0\\  
    0   & \,\,0   & \,\, 0 \\
   0 & \,\,0 &\,\, 0
  \end{array}
\right).
\end{equation*}
Then, for any $0<\delta<1$ and $\epsilon_0>0$, there are $X,Y,Z\in\mathbb S^n$ satisfying
\begin{equation*}
\left(         
  \begin{array}{ccc}   
    X  &   0  &  0\\  
    0   & Y &  0\\
   0   &  0  &  -Z
  \end{array}
\right)-\left(         
  \begin{array}{ccc}   
    g_2(\hat x)  &   0  &   0\\  
    0   &   0   &  0\\
    0  &    0   &  0
  \end{array}
\right)
\leq (1+\epsilon_0)\left[g_0(\hat x,\hat y,\hat z)\left(         
  \begin{array}{ccc}   
    I  &   -I  & 0\\  
    -I   & I   &  0 \\
   0 & 0 & 0
  \end{array}
\right)+g_1(\hat x,\hat y,\hat z)
\left(         
  \begin{array}{ccc}   
    I  &   I & -2I\\  
    I   & I  &  -2I\\
   -2I  &-2I & 4I 
  \end{array}
\right)\right],
\end{equation*}
such that
\begin{equation*}
G\big(\hat x,u(\hat x),D_x\psi(\hat x,\hat y,\hat z),X,I^{1,\delta}[\hat x,D_x\psi(\hat x,\hat y,\hat z),\psi(\cdot,\hat y,\hat z)]+I^{2,\delta}[\hat x,D_x\psi(\hat x,\hat y,\hat z),u(\cdot)]\big)\leq 0,
\end{equation*}
\begin{equation*}
G\big(\hat y,v(\hat y),D_y\psi(\hat x,\hat y,\hat z),Y,I^{1,\delta}[\hat y,D_y\psi(\hat x,\hat y,\hat z),\psi(\hat x,\cdot,\hat z)]+I^{2,\delta}[\hat y,D_y\psi(\hat x,\hat y,\hat z),v(\cdot)]\big)\leq 0,
\end{equation*}
\begin{equation*}
G\big(\hat z,w(\hat z),-\frac{1}{2}D_z\psi(\hat x,\hat y,\hat z),\frac{Z}{2},I^{1,\delta}[\hat z,-\frac{D_z\psi(\hat x,\hat y,\hat z)}{2},-\frac{\psi(\hat x,\hat y,\cdot)}{2}]+I^{2,\delta}[\hat z,-\frac{D_z\psi(\hat x,\hat y,\hat z)}{2},w(\cdot)]\big)\geq 0.
\end{equation*}
\end{lemma}
\begin{proof}
This lemma can be deduced from the proof of Theorem 4.9 in \cite{jk1}.
\end{proof}
\begin{remark}\label{re4.1}
Lemma $\ref{le4.1}$ is also true for Bellman-Isaacs equations ($\ref{eq..1.6}$).
\end{remark}
\begin{lemma}\label{le..4.2}
Suppose that a L\'{e}vy measure $\mu$ satisfies ($\ref{eq:1.5}$) and $j(x,\xi)$ satisfies assumption ($\bar{H}2$). Then
\begin{eqnarray}
M_2:&=&\sup_{\phi(x,y,z)\not=0}\Big\{\phi(x,y,z)^{-1}\int_{\mathbb R^n}\Big[\phi(x+j(x,\xi),y+j(y,\xi),z+j(z,\xi))-\phi(x,y,z)\nonumber\\
&&-\mathbbm{1}_{B_1(0)}(\xi)\Big(D_x\phi(x,y,z),D_y\phi(x,y,z),D_z\phi(x,y,z)\Big)\cdot \Big(j(x,\xi),j(y,\xi),j(z,\xi)\Big)\Big]\mu(d\xi)\Big\}\nonumber\\
&<&+\infty,\nonumber\\\label{eq,4.3}
\end{eqnarray}
where $\phi(x,y,z)$ is defined in ($\bar{H}1$).
\end{lemma}
\begin{proof}
By direct calculations, we have 
\begin{equation}\label{eq4.5}
D\phi(x,y,z)=\frac{1}{\phi(x,y,z)}\left[\bar{\theta}|x-y|^{2\bar{\theta}-2}
\left(         
  \begin{array}{c}   
    x-y   \\  
    y-x   \\
    0  
  \end{array}
\right)
+\left(         
  \begin{array}{c}   
    x+y-2z   \\  
    x+y-2z  \\
   -2x-2y+4z  
  \end{array}
\right)
\right]
\end{equation}
and 
\begin{eqnarray}
D^2\phi(x,y,z)&=&-\frac{1}{\phi(x,y,z)}D\phi(x,y,z)\otimes D\phi(x,y,z)+\frac{1}{\phi(x,y,z)}
\Bigg[\bar{\theta}|x-y|^{2\bar{\theta}-2}
\left(         
  \begin{array}{ccc}   
    I & -I &\,\, 0  \\  
    -I& I &\,\, 0\\
   0 & 0 &\,\, 0
  \end{array}
\right)\nonumber\\
&&+\bar\theta(2\bar\theta-2)|x-y|^{2\bar\theta-4}
\left(         
  \begin{array}{c}   
    x- y   \\  
    y- x   \\
    0  
  \end{array}
\right)\otimes
\left(         
  \begin{array}{c}   
    x-y   \\  
    y-x   \\
    0  
  \end{array}
\right)+
\left(         
  \begin{array}{ccc}   
    I & I & -2I  \\  
    I& I &   -2I\\
   -2I & -2I &  4I
  \end{array}
\right)
\Bigg]\nonumber\\
&\leq&\frac{1}{\phi(x,y,z)}\left[\bar{\theta}(2\bar\theta-1)|x-y|^{2\bar{\theta}-2}
\left(         
  \begin{array}{ccc}   
    I & -I &\,\, 0  \\  
    -I& I &\,\, 0\\
   0 & 0 &\,\, 0
  \end{array}
\right)+
\left(         
  \begin{array}{ccc}   
    I & I & -2I  \\  
    I& I &   -2I\\
   -2I & -2I &  4I
  \end{array}
\right)\right].\label{eq4.6}
\end{eqnarray}
Since $\lim_{\xi\to 0}\rho(\xi)=0$, there exists a positive constant $\delta_2<1$ such that $\sup_{\xi\in B_{\delta_2}(0)}\rho(\xi)\leq \frac{1}{4}$. By (\ref{eq4.5}) and (\ref{eq4.6}), we have, for any $x,y,z\in\mathbb R^n$ and $\phi(x,y,z)\not=0$,
\begin{eqnarray*}
&&\phi(x,y,z)^{-1}\int_{\mathbb R^n}\Big[\phi(x+j(x,\xi),y+j(y,\xi),z+j(z,\xi))-\phi(x,y,z)\nonumber\\
&&-\mathbbm{1}_{B_1(0)}(\xi)\Big(D_x\phi(x,y,z),D_y\phi(x,y,z),D_z\phi(x,y,z)\Big)\cdot \Big(j(x,\xi),j(y,\xi),j(z,\xi)\Big)\Big]\mu(d\xi)\nonumber\\
&\leq&\phi(x,y,z)^{-1}\Big\{\int_{B_{\delta_2}(0)}\Big[\sup_{0\leq t\leq 1}\Big(j(x,\xi),j(y,\xi),j(z,\xi)\Big)D^2\phi\big(x+tj(x,\xi),y+tj(y,\xi),z+tj(z,\xi)\big)\nonumber\\
&&\Big(j(x,\xi),j(y,\xi),j(z,\xi)\Big)^T\Big]\mu(d\xi)\nonumber\\
&&+\int_{B_{\delta_2}^c(0)}\Big[\Big(|x-y+j(x,\xi)-j(y,\xi)|^{2\bar\theta}+|x+y-2z+j(x,\xi)+j(y,\xi)-2j(z,\xi)|^2\Big)^{\frac{1}{2}}\nonumber\\
&&-\phi(x,y,z)-\mathbbm{1}_{B_1(0)}(\xi)\frac{1}{\phi(x,y,z)}\nonumber\\
&&\Big(\bar\theta|x-y|^{2\bar\theta-2}(x-y)\cdot \big(j(x,\xi)-j(y,\xi)\big)+(x+y-2z)\cdot\big(j(x,\xi)+j(y,\xi)-2j(z,\xi)\big)\Big)\Big]\mu(d\xi)\Big\}\nonumber
\end{eqnarray*}
\begin{eqnarray*}
&\leq&\phi(x,y,z)^{-1}\Big\{\int_{B_{\delta_2}(0)}\Big[\sup_{0\leq t\leq 1}\frac{1}{\phi(x+tj(x,\xi),y+tj(y,\xi),z+tj(z,\xi))}\Big(\big(j(x,\xi)+j(y,\xi)-2j(z,\xi)\big)^2\nonumber\\
&&+\bar\theta(2\bar\theta-1)|x-y+t\big(j(x,\xi)-j(y,\xi)\big)|^{2\bar\theta-2}\big(j(x,\xi)-j(y,\xi)\big)^2\Big)\Big]\mu(d\xi)\nonumber\\
&&+\int_{B_{\delta_2}^c(0)}\Big[\Big(|x-y+j(x,\xi)-j(y,\xi)|^{2\bar\theta}+|x+y-2z+j(x,\xi)+j(y,\xi)-2j(z,\xi)|^2\Big)^{\frac{1}{2}}\nonumber\\
&&-\phi(x,y,z)-\mathbbm{1}_{B_1(0)}(\xi)\frac{1}{\phi(x,y,z)}\nonumber\\
&&\Big(\bar\theta|x-y|^{2\bar\theta-2}(x-y)\cdot \big(j(x,\xi)-j(y,\xi)\big)+(x+y-2z)\cdot\big(j(x,\xi)+j(y,\xi)-2j(z,\xi)\big)\Big)\Big]\mu(d\xi)\Big\}.\nonumber
\end{eqnarray*}
By ($\bar{H}2$), we have
\begin{eqnarray*}
|j(x,\xi)+j(y,\xi)-2j(z,\xi)|&\leq& |j(x,\xi)+j(y,\xi)-2j(\frac{x+y}{2})|+|2j(\frac{x+y}{2},\xi)-2j(z,\xi)|\\
&\leq&\rho(\xi)\big(|x-y|^{\bar\theta}+|x+y-2z|\big).
\end{eqnarray*}
Using it, we obtain, for any $\xi\in B_{\delta_2}(0)$ and $t\in[0,1]$,
\begin{eqnarray*}
&&\phi(x+tj(x,\xi),y+tj(y,\xi),z+j(z,\xi))\\
&=&\Big[|x-y+t\big(j(x,\xi)-j(y,\xi)\big)|^{2\bar\theta}+|x+y-2z+t\big(j(x,\xi)+j(y,\xi)-2j(z,\xi)\big)|^2\Big]^{\frac{1}{2}}\\
&\geq&\Big[(\frac{3}{4})^{2\bar\theta}|x-y|^{2\bar\theta}+\big(\frac{3}{4}|x+y-2z|-\frac{1}{4}|x-y|^{\bar\theta}\big)^2\Big]^{\frac{1}{2}}\\
&\geq&\Big\{\big[(\frac{3}{4})^{2\bar\theta}-\frac{1}{16}\big]|x-y|^{2\bar\theta}+\frac{9}{32}|x+y-2z|^2\Big\}^{\frac{1}{2}}\\
&\geq&\frac{1}{2}\phi(x,y,z).
\end{eqnarray*}
Therefore, for any $x,y,z\in\mathbb R^n$ and $\phi(x,y,z)\not=0$, we have by ($\ref{eq:1.5}$),
\begin{eqnarray}
&&\phi(x,y,z)^{-1}\int_{\mathbb R^n}\Big[\phi(x+j(x,\xi),y+j(y,\xi),z+j(z,\xi))-\phi(x,y,z)\nonumber\\
&&-\mathbbm{1}_{B_1(0)}(\xi)\Big(D_x\phi(x,y,z),D_y\phi(x,y,z),D_z\phi(x,y,z)\Big)\cdot \Big(j(x,\xi),j(y,\xi),j(z,\xi)\Big)\Big]\mu(d\xi)\nonumber\\
&\leq&2\int_{B_{\delta_2}(0)}\Big[2+(\frac{5}{4})^{2\bar\theta-2}\bar\theta(2\bar\theta-1)\Big]\rho(\xi)^2\mu(d\xi)+\int_{B_{\delta_2}^c(0)}\Big\{\sqrt{2}\Big[\big(1+\rho(\xi)\big)^{\bar \theta}+\rho(\xi)\Big]-1\Big\}\mu(d\xi)\nonumber\\
&&+(\bar \theta+\frac{3}{2})\int_{B_1(0)\cap B_{\delta_2}^c(0)}\rho(\xi)\mu(d\xi)<+\infty.\label{eq,,4.12}
\end{eqnarray}
\end{proof}
\begin{theorem}\label{th3.1}
Suppose that the nonlinearity $G$ in ($\ref{eq1.1}$) is continuous, and satisfies ($\ref{eq:1.5}$)-($\ref{eq1.2}$) and ($\bar{H}1$). Suppose that $j(x,\xi)$ satisfies assumption ($\bar{H}2$). Then, if $u\in C^{0,1}(\bar{\mathbb R^n})$ is a viscosity solution of ($\ref{eq1.1}$) and $\gamma>\Lambda M_2+C_3$ where $M_2$ is defined in ($\ref{eq,4.3}$), then $u$ is $\bar \theta$-semiconcave in $\mathbb R^n$.
\end{theorem}
\begin{proof}
Let $\Phi(x,y,z)=u(x)+u(y)-2u(z)-\psi(x,y,z)$ where $\psi(x,y,z)=L\phi(x,y,z)+\eta|x|^2$ and $\phi(x,y,z)$ is defined in ($\bar{H}1$). We want to prove, for any $\eta>0$, we have $\Phi(x,y,z)\leq 0$ for all $x,y,z\in\mathbb R^n$ and some fixed  sufficiently large $L$. Otherwise, there exists a positive constant $\eta_0$ such that $\sup_{x,y,z\in\mathbb R^n}\Phi(x,y,z)>0$ if $0<\eta<\eta_0$. By boundedness of $u$, there is a point ($\hat x,\hat y,\hat z$) such that $\Phi(\hat x,\hat y,\hat z)=\sup_{x,y,z\in\mathbb R^n}\Phi(x,y,z)>0$. Therefore, we have 
\begin{equation}\label{eq4.2}
\max\{\eta|\hat x|^2,L\phi(\hat x,\hat y,\hat z)\}< u(\hat x)+u(\hat y)-2u(\hat z).
\end{equation}
By ($\ref{eq4.5}$) and ($\ref{eq4.6}$), we have 
\begin{eqnarray*}
D^2\psi(\hat x,\hat y,\hat z)&\leq&\frac{L}{\phi(\hat x,\hat y,\hat z)}\left[\bar{\theta}(2\bar\theta-1)|\hat x-\hat y|^{2\bar{\theta}-2}
\left(         
  \begin{array}{ccc}   
    I & -I &\,\, 0  \\  
    -I& I &\,\, 0\\
   0 & 0 &\,\, 0
  \end{array}
\right)+
\left(         
  \begin{array}{ccc}   
    I & I & -2I  \\  
    I& I &   -2I\\
   -2I & -2I &  4I
  \end{array}
\right)\right]\\
&&+2\eta\left(         
  \begin{array}{ccc}   
    I &\,\, 0 &\,\, 0  \\  
    0&\,\, 0 &\,\, 0\\
   0 &\,\, 0 &\,\, 0
  \end{array}
\right).
\end{eqnarray*}
By Lemma $\ref{le4.1}$, since $u\in BUC(\mathbb R^n)$ is a viscosity solution of ($\ref{eq1.1}$), for any $0<\delta<1$ and $\epsilon_0>0$, there are $X,Y,Z\in\mathbb S^n$ satisfying 
\begin{eqnarray}
\left(         
  \begin{array}{ccc}   
    X  &   0  &  0\\  
    0   & Y &  0\\
   0   &  0  &  -Z
  \end{array}
\right)-2\eta\left(         
  \begin{array}{ccc}   
    I&\,\,   0  &   \,\,0\\  
    0   & \,\,  0   &  \,\,0\\
    0  &  \,\,  0   &  \,\,0
  \end{array}
\right)&\leq&\frac{(1+\epsilon_0)L}{\phi(\hat x,\hat y,\hat z)}\Bigg[\bar{\theta}(2\bar\theta-1)|\hat x-\hat y|^{2\bar{\theta}-2}
\left(         
  \begin{array}{ccc}   
    I & -I &\,\, 0  \\  
    -I& I &\,\, 0\\
   0 & 0 &\,\, 0
  \end{array}
\right)\nonumber\\
&&+
\left(         
  \begin{array}{ccc}   
    I & I & -2I  \\  
    I& I &   -2I\\
   -2I & -2I &  4I
  \end{array}
\right)\Bigg],\label{eq:4.4}
\end{eqnarray}
such that
\begin{equation*}
G(\hat x,u(\hat x),LD_x\phi(\hat x,\hat y,\hat z)+2\eta\hat x,X,l_{\hat x})\leq 0,
\end{equation*}
\begin{equation*}
G(\hat y,u(\hat y),LD_y\phi(\hat x,\hat y,\hat z),Y,l_{\hat y})\leq 0,
\end{equation*}
\begin{equation*}
G(\hat z,u(\hat z),-\frac{L}{2}D_z\phi(\hat x,\hat y,\hat z),\frac{Z}{2},l_{\hat z})\geq 0,
\end{equation*}
where 
\begin{eqnarray*}
&&l_{\hat x}=I^{1,\delta}[\hat x,LD_x\phi(\hat x,\hat y,\hat z)+2\eta\hat x,L\phi(\cdot,\hat y,\hat z)+\eta|\cdot|^2]+I^{2,\delta}[\hat x,LD_x\phi(\hat x,\hat y,\hat z)+2\eta\hat x,u(\cdot)],\\
&&l_{\hat y}=I^{1,\delta}[\hat y,LD_y\phi(\hat x,\hat y,\hat z),L\phi(\hat x,\cdot,\hat z)]+I^{2,\delta}[\hat y,LD_y\phi(\hat x,\hat y,\hat z),u(\cdot)],\\
&&l_{\hat z}=I^{1,\delta}[\hat z,-\frac{L}{2}D_z\phi(\hat x,\hat y,\hat z),-\frac{L}{2}\phi(\hat x,\hat y,\cdot)]+I^{2,\delta}[\hat z,-\frac{L}{2}D_z\phi(\hat x,\hat y,\hat z),u(\cdot)].
\end{eqnarray*}
Therefore, by ($\bar{H}1$) and ($\ref{eq4.2}$), we have
\begin{equation}\label{eq:4.5}
\gamma L\phi(\hat x,\hat y,\hat z)\leq \Lambda(l_{\hat x}+l_{\hat y}-2l_{\hat z})+C_3(1+L)\phi(\hat x,\hat y,\hat z)+C_4\eta(1+|\hat x|^2).
\end{equation}
We now estimate the integral term $l_{\hat x}+l_{\hat y}-2l_{\hat z}$.
\begin{eqnarray*}
&&l_{\hat x}+l_{\hat y}-2l_{\hat z}\\
&=&L\int_{B_{\delta}(0)}\Big(\phi(\hat x+j(\hat x,\xi),\hat y,\hat z)-\phi(\hat x,\hat y,\hat z)-D_x\phi(\hat x,\hat y,\hat z)\cdot j(\hat x,\xi)\Big)\mu(d\xi)\\
&&+\eta\int_{B_{\delta}(0)}\Big(|\hat x+j(\hat x,\xi)|^2-|\hat x|^2-2\hat x\cdot j(\hat x,\xi)\Big)\mu(d\xi)\\
&&+L\int_{B_{\delta}(0)}\Big(\phi(\hat x,\hat y+j(\hat y,\xi),\hat z)-\phi(\hat x,\hat y,\hat z)-D_y\phi(\hat x,\hat y,\hat z)\cdot j(\hat y,\xi)\Big)\mu(d\xi)\\
&&+L\int_{B_{\delta}(0)}\Big(\phi(\hat x,\hat y,\hat z+j(\hat z,\xi))-\phi(\hat x,\hat y,\hat z)-D_z\phi(\hat x,\hat y,\hat z)\cdot j(\hat z, \xi)\Big)\mu(d\xi)
\end{eqnarray*}
\begin{eqnarray*}
&&+\int_{B_\delta^c(0)}\Big[u(\hat x+j(\hat x,\xi))-u(\hat x)+u(\hat y+j(\hat y,\xi))-u(\hat y)-2\big(u(\hat z+j(\hat z,\xi))-u(\hat z)\big)\\
&&-\mathbbm{1}_{B_1(0)}(\xi)\big(LD_x\phi(\hat x,\hat y,\hat z)+2\eta\hat x\big)\cdot j(\hat x,\xi)-\mathbbm{1}_{B_1(0)}(\xi)LD_y\phi(\hat x,\hat y,\hat z)\cdot j(\hat y,\xi)\\
&&-\mathbbm{1}_{B_1(0)}(\xi)LD_z\phi(\hat x,\hat y,\hat z)\cdot j(\hat z,\xi)\Big]\mu(d\xi).
\end{eqnarray*}
Thus, by ($\ref{eq4.5}$) and ($\ref{eq4.6}$), we have
\begin{eqnarray*}
&&l_{\hat x}+l_{\hat y}-2l_{\hat z}\\
&\leq&L\int_{B_\delta(0)}\Big[\sup_{0\leq t\leq 1}\frac{1}{\phi(\hat x+tj(\hat x,\xi),\hat y,\hat z)}\Big(\bar\theta(2\bar\theta-1)|\hat x-\hat y+tj(\hat x,\xi)|^{2\bar\theta-2}+1\Big)|j(\hat x,\xi)|^2\Big]\mu(d\xi)\\
&&+L\int_{B_\delta(0)}\Big[\sup_{0\leq t\leq 1}\frac{1}{\phi(\hat x,\hat y+tj(\hat y,\xi),\hat z)}\Big(\bar\theta(2\bar\theta-1)|\hat x-\hat y-tj(\hat y,\xi)|^{2\bar\theta-2}+1\Big)|j(\hat y,\xi)|^2\Big]\mu(d\xi)\\
&&+4L\int_{B_\delta(0)}\Big(\sup_{0\leq t\leq 1}\frac{1}{\phi(\hat x,\hat y,\hat z+tj(\hat z,\xi))}|j(\hat z,\xi)|^2\Big)\mu(d\xi)+\eta\int_{B_\delta(0)}|j(\hat x,\xi)|^2\mu(d\xi)\\
&&+\int_{B_\delta^c(0)}\Big[u(\hat x+j(\hat x,\xi))-u(\hat x)+u(\hat y+j(\hat y,\xi))-u(\hat y)-2\big(u(\hat z+j(\hat z,\xi))-u(\hat z)\big)\\
&&-\mathbbm{1}_{B_1(0)}(\xi)L\Big(D_x\phi(\hat x,\hat y,\hat z),D_y\phi(\hat x,\hat y,\hat z),D_z\phi(\hat x,\hat y,\hat z)\Big)\cdot \Big(j(\hat x,\xi),j(\hat y,\xi),j(\hat z,\xi)\Big)\\
&&-\mathbbm{1}_{B_1(0)}(\xi)2\eta\hat x\cdot j(\hat x,\xi)\Big]\mu(d\xi).
\end{eqnarray*}
Since $\Phi(x,y,z)$ attains a global maximum at ($\hat x,\hat y,\hat z$), we have, for any $\xi\in\mathbb R^n$,
\begin{eqnarray}\label{eq4.7}
&&u(\hat x+j(\hat x,\xi))-u(\hat x)+u(\hat y+j(\hat y,\xi))-u(\hat y)-2\big(u(\hat z+j(\hat z,\xi))-u(\hat z)\big)\nonumber\\
&\leq&L\phi(\hat x+j(\hat x,\xi),\hat y+j(\hat y,\xi),\hat z+j(\hat z,\xi))-L\phi(\hat x,\hat y,\hat z)+\eta|\hat x+j(\hat x,\xi)|^2-\eta|\hat x|^2.
\end{eqnarray}
By ($\ref{eq4.7}$), we have
\begin{eqnarray}\label{eq::::4.9}
&&l_{\hat x}+l_{\hat y}-2l_{\hat z}\nonumber\\
&\leq&L\int_{B_\delta(0)}\Big[\sup_{0\leq t\leq 1}\frac{1}{\phi(\hat x+tj(\hat x,\xi),\hat y,\hat z)}\Big(\bar\theta(2\bar\theta-1)|\hat x-\hat y+tj(\hat x,\xi)|^{2\bar\theta-2}+1\Big)|j(\hat x,\xi)|^2\Big]\mu(d\xi)\nonumber\\
&&+L\int_{B_\delta(0)}\Big[\sup_{0\leq t\leq 1}\frac{1}{\phi(\hat x,\hat y+tj(\hat y,\xi),\hat z)}\Big(\bar\theta(2\bar\theta-1)|\hat x-\hat y-tj(\hat y,\xi)|^{2\bar\theta-2}+1\Big)|j(\hat y,\xi)|^2\Big]\mu(d\xi)\nonumber\\
&&+4L\int_{B_\delta(0)}\Big(\sup_{0\leq t\leq 1}\frac{1}{\phi(\hat x,\hat y,\hat z+tj(\hat z,\xi))}|j(\hat z,\xi)|^2\Big)\mu(d\xi)\nonumber\\
&&+\eta\int_{\mathbb R^n}\Big(|\hat x+j(\hat x,\xi)|^2-|\hat x|^2-\mathbbm{1}_{B_1(0)}(\xi)2\hat x\cdot j(\hat x,\xi)\Big)\mu(d\xi)\nonumber\\
&&+L\int_{B_\delta^c(0)}\Big[\phi(\hat x+j(\hat x,\xi),\hat y+j(\hat y,\xi),\hat z+j(\hat z,\xi))-\phi(\hat x,\hat y,\hat z)\nonumber\\
&&-\mathbbm{1}_{B_1(0)}(\xi)\Big(D_x\phi(\hat x,\hat y,\hat z),D_y\phi(\hat x,\hat y,\hat z),D_z\phi(\hat x,\hat y,\hat z)\Big)\cdot \Big(j(\hat x,\xi),j(\hat y,\xi),j(\hat z,\xi)\Big)\Big]\mu(d\xi).
\end{eqnarray}
Similarly as in the proof of Theorem \ref{th2.1}, we have $\eta|\hat x|^2\to 0$ as $\eta\to 0$ and
\begin{equation*}
\epsilon_1\leq \phi(\hat x,\hat y,\hat z)\leq \epsilon_1^{-1},
\end{equation*}
where $\epsilon_1$ is a positive constant independent of $\eta$. Letting $\delta\rightarrow 0$ and then letting $\eta\rightarrow 0$ in ($\ref{eq:4.5}$), we have, by ($\ref{eq:1.5}$), ($\ref{eq::::4.9}$) and ($\bar{H}2$),
\begin{eqnarray*}
\gamma &L&\phi(\hat x,\hat y,\hat z)\leq \Lambda L\int_{\mathbb R^n}\Big[\phi(\hat x+j(\hat x,\xi),\hat y+j(\hat y,\xi),\hat z+j(\hat z,\xi))-\phi(\hat x,\hat y,\hat z)\\
&&-\mathbbm{1}_{B_1(0)}(\xi)\Big(D_x\phi(\hat x,\hat y,\hat z),D_y\phi(\hat x,\hat y,\hat z),D_z\phi(\hat x,\hat y,\hat z)\Big)\cdot \Big(j(\hat x,\xi),j(\hat y,\xi),j(\hat z,\xi)\Big)\Big]\mu(d\xi)\\
&&+C_3(1+L)\phi(\hat x,\hat y,\hat z).
\end{eqnarray*}
Therefore, by Lemma $\ref{le..4.2}$,
\begin{eqnarray}
\gamma&\leq&\Lambda\phi(\hat x,\hat y,\hat z)^{-1}\int_{\mathbb R^n}\Big[\phi(\hat x+j(\hat x,\xi),\hat y+j(\hat y,\xi),\hat z+j(\hat z,\xi))-\phi(\hat x,\hat y,\hat z)\nonumber\\
&&-\mathbbm{1}_{B_1(0)}(\xi)\Big(D_x\phi(\hat x,\hat y,\hat z),D_y\phi(\hat x,\hat y,\hat z),D_z\phi(\hat x,\hat y,\hat z)\Big)\cdot \Big(j(\hat x,\xi),j(\hat y,\xi),j(\hat z,\xi)\Big)\Big]\mu(d\xi)\nonumber\\
&&+C_3(1+\frac{1}{L})\nonumber\\
&\leq&\Lambda M_2+C_3(1+\frac{1}{L})<+\infty,\label{eq,,4.11}
\end{eqnarray}
where $M_2$ is defined in ($\ref{eq,4.3}$). This yields a contradiction, if $\gamma>\Lambda M_2+C_3$, for sufficiently large $L$. Therefore, $u$ is $\bar\theta$-semiconcave in $\mathbb R^n$.
\end{proof}
Let us consider the semiconcavity of viscosity solutions of the Bellman equation ($\ref{eq1.7}$). The following estimates will be frequently used in the proof of the semiconcavity.
\begin{lemma}\label{lemm4.2}
(a) If $f$ is $\bar\theta$-semiconvex with constant $C$ in $\mathbb R^n$ and $[f]_{0,1;\mathbb R^n}<+\infty$, then 
\begin{equation*}
2f(z)-f(x)-f(y)\leq C|x-y|^{\bar\theta}+[f]_{0,1;\mathbb R^n}|x+y-2z|.
\end{equation*}
Moreover, if $[f]_{1,\bar\theta-1;\mathbb R^n}<+\infty$, then
\begin{equation*}
|f(x)+f(y)-2f(z)|\leq \frac{\sqrt{n}}{2}[f]_{1,\bar\theta-1;\mathbb R^n}|x-y|^{\bar\theta}+[f]_{0,1;\mathbb R^n}|x+y-2z|.
\end{equation*}
(b) If $f\in C^{0,1}(\bar{\mathbb R}^n)$, then
\begin{equation*}
|f(x)-f(z)|\leq2\max\{|f|_{0;\mathbb R^n},[f]_{0,1;\mathbb R^n}\}\phi(x,y,z)^{\frac{1}{2}},
\end{equation*}
where $\phi(x,y,z)$ is defined in ($\bar{H}1$).
\end{lemma}
\begin{proof}
(a) Since $f$ is $\bar\theta$-semiconvex with constant $C$ in $\mathbb R^n$ and $[f]_{0,1;\mathbb R^n}<+\infty$,
\begin{eqnarray*}
2f(z)-f(x)-f(y)&=&2f(\frac{x+y}{2})-f(x)-f(y)+\Big(2f(z)-2f(\frac{x+y}{2})\Big)\\
&\leq& C|x-y|^{\bar\theta}+[f]_{0,1;\mathbb R^n}| x+y-2 z|.
\end{eqnarray*}
Moreover, if $[f]_{1,\bar\theta-1;\mathbb R^n}<+\infty$, then $f$ is $\bar\theta$-semiconvex and $\bar\theta$-semiconcave with a constant $\frac{\sqrt{n}}{2}[f]_{1,\bar\theta-1;\mathbb R^n}$ in $\mathbb R^n$. Thus, the result follows from the above estimate.\\
(b) Since $g\in C^{0,1}(\bar{\mathbb R}^n)$, then
\begin{eqnarray*}
|g(x)-g(z)|&\leq&|g(x)-g(\frac{x+y}{2})|+|g(\frac{x+y}{2})-g(z)|\\
&\leq&[g]_{0,1;\mathbb R^n}|\frac{x-y}{2}|+\Big(2|g|_{0;\mathbb R^n}[g]_{0,1;\mathbb R^n}\frac{|x+y-2z|}{2}\Big)^{\frac{1}{2}}\\
&\leq&2\max\{|g|_{0;\mathbb R^n},[g]_{0,1;\mathbb R^n}\}\phi(x,y,z)^{\frac{1}{2}}.
\end{eqnarray*}
\end{proof}
\begin{theorem}\label{th4.2}
Suppose that $c_{\alpha}\geq\gamma$ in $\mathbb R^n$ uniformly in $\alpha\in\mathcal{A}$. There exist a positive constant $C$ and $1<\bar\theta\leq 2$ such that ($\ref{eq::3.8}$) holds and
\begin{equation}\label{eq::4.8}
\sup_{\alpha\in\mathcal{A}}\max\{[\sigma_{\alpha}]_{0,1;\mathbb R^n},[\sigma_{\alpha}]_{1,\bar\theta-1;\mathbb R^n},[b_{\alpha}]_{0,1;\mathbb R^n},[b_{\alpha}]_{1,\bar\theta-1,\mathbb R^n},[f_{\alpha}]_{0,1,\mathbb R^n}\}<+\infty.
\end{equation} 
Suppose that the L\'evy measure $\mu$ satisfies ($\ref{eq:1.5}$), the family $\{j_{\alpha}(x,\xi)\}$ satisfies assumption ($\bar{H}2$) uniformly in $\alpha\in \mathcal A$, and $c_{\alpha}\in C^{1,\bar\theta-1}(\bar{\mathbb R}^n)$ and $\{f_\alpha\}$ is uniformly $\bar\theta$-semiconvex with constant $C_5$, uniformly in $\alpha\in\mathcal{A}$. Then, if $u\in C^{0,1}(\bar{\mathbb R}^n)$ is a viscosity solution of ($\ref{eq1.7}$) and $\gamma>N_2$ where 
\begin{eqnarray}\label{eq,4.13}
N_2:&=&\sup_{\phi(x,y,z)\not=0}\sup_{\alpha\in\mathcal{A}}\phi(x,y,z)^{-2}\Big\{\bar\theta(2\bar\theta-1)|x-y|^{2\bar\theta-2}Tr\Big[\big(\sigma_{\alpha}(x)-\sigma_{\alpha}(y)\big)\big(\sigma_{\alpha}(x)-\sigma_{\alpha}(y)\big)^T\Big]\nonumber\\
&&+Tr\Big[\big(\sigma_{\alpha}(x)+\sigma_{\alpha}(y)-2\sigma_{\alpha}(z)\big)\big(\sigma_{\alpha}(x)+\sigma_{\alpha}(y)-2\sigma_{\alpha}(z)\big)^T\Big]\nonumber\\
&&+\bar\theta|x-y|^{2\bar\theta-2}(x-y)\cdot\big(b_{\alpha}(y)-b_{\alpha}(x)\big)+(x+y-2z)\cdot\big(2b_{\alpha}(z)-b_{\alpha}(x)-b_{\alpha}(y)\big)\nonumber\\
&&+\phi(x,y,z)\int_{\mathbb R^n}\Big[\phi\big(x+j_\alpha(x,\xi),y+j_\alpha(y,\xi),z+j_\alpha(z,\xi)\big)-\phi(x,y,z)\nonumber\\
&&-\mathbbm{1}_{B_1(0)}(\xi)\Big(D_x\phi(x,y,z),D_y\phi(x,y,z),D_z\phi(x,y,z)\Big)\cdot \Big(j_\alpha(x,\xi),j_\alpha(y,\xi),j_\alpha(z,\xi)\Big)\Big]\mu(d\xi)\Big\}\nonumber\\
&<&+\infty,
\end{eqnarray}
then $u$ is $\bar\theta$-semiconcave in $\mathbb R^n$.
\end{theorem}
\begin{proof}
At the beginning of the proof, we will show that the constant $N_2$ has an upper bound. By ($\ref{eq::4.8}$), Lemma \ref{lemm4.2} and the estimates in ($\ref{eq,,4.12}$), we have
\begin{eqnarray*}
N_2&\leq&\bar\theta(2\bar\theta-1)\sup_{\alpha\in\mathcal{A}}[\sigma_\alpha]_{0,1;\mathbb R^n}^2+\big(\frac{\sqrt{n}}{2}\sup_{\alpha\in\mathcal{A}}[\sigma_\alpha]_{1,\bar\theta-1;\mathbb R^n}+\sup_{\alpha\in\mathcal{A}}[\sigma_\alpha]_{0,1;\mathbb R^n}\big)^2+\bar\theta\sup_{\alpha\in\mathcal{A}}[b_\alpha]_{0,1;\mathbb R^n}\\
&&+\big(\frac{\sqrt{n}}{2}\sup_{\alpha\in\mathcal{A}}[b_\alpha]_{1,\bar\theta-1;\mathbb R^n}+\sup_{\alpha\in\mathcal{A}}[b_\alpha]_{0,1;\mathbb R^n}\big)+2\int_{B_{\delta_2}(0)}\Big[2+(\frac{5}{4})^{2\bar\theta-2}\bar\theta(2\bar\theta-1)\Big]\rho(\xi)^2\mu(d\xi)\\
&&+\int_{B_{\delta_2}^c(0)}\Big\{\sqrt{2}\Big[(1+\rho(\xi))^{\bar \theta}+\rho(\xi)\Big]-1\Big\}\mu(d\xi)+(\bar \theta+\frac{3}{2})\int_{B_1(0)\cap B_{\delta_2}^c(0)}\rho(\xi)\mu(d\xi)<+\infty,
\end{eqnarray*}
where $\delta_2$ was chosen in Lemma $\ref{le..4.2}$.

Then we want to prove that, for any $\eta>0$, $\Phi(x,y,z)=u(x)+u(y)-2u(z)-\psi(x,y,z)\leq 0$ for all $x,y,z\in\mathbb R^n$ and some fixed  sufficiently large $L$, where $\psi(x,y,z)$ is given in Theorem $\ref{th3.1}$. Otherwise, there exists a positive constant $\eta_0$ such that $\sup_{x,y,z\in\mathbb R^n}\Phi(x,y,z)>0$ if $0<\eta<\eta_0$. By boundedness of $u$, there is a point ($\hat x,\hat y,\hat z$) such that $\Phi(\hat x,\hat y,\hat z)=\sup_{x,y,z\in\mathbb R^n}\Phi(x,y,z)>0$. Therefore, we have ($\ref{eq4.2}$). By Remark $\ref{re4.1}$, since $u\in BUC(\mathbb R^n)$ is a viscosity solution of ($\ref{eq1.7}$), we have, for any $0<\delta<1$ and $\epsilon_0>0$, there are $X,Y,Z\in \mathbb S^n$ satisfying ($\ref{eq:4.4}$) such that
\begin{equation*}
\sup_{\alpha\in\mathcal{A}}\big\{-Tr\big(\sigma_{\alpha}(\hat x)\sigma_{\alpha}^T(\hat x)X\big)-l_{\hat x,\alpha}+b_\alpha(\hat x)\cdot D_x\psi(\hat x,\hat y,\hat z)+c_{\alpha}(\hat x)u(\hat x)+f_{\alpha}(\hat x)\big\}\leq 0,
\end{equation*}
\begin{equation*}
\sup_{\alpha\in\mathcal{A}}\big\{-Tr\big(\sigma_{\alpha}(\hat y)\sigma_{\alpha}^T(\hat y)Y\big)-l_{\hat y,\alpha}+b_\alpha(\hat y)\cdot D_y\psi(\hat x,\hat y,\hat z)+c_{\alpha}(\hat y)u(\hat y)+f_{\alpha}(\hat y)\big\}\leq 0,
\end{equation*}
\begin{equation*}
\sup_{\alpha\in\mathcal{A}}\big\{-Tr\big(\sigma_{\alpha}(\hat z)\sigma_{\alpha}^T(\hat z)\frac{Z}{2}\big)-l_{\hat z,\alpha}-b_\alpha(\hat z)\cdot\frac{D_z\psi(\hat x,\hat y,\hat z)}{2}+c_{\alpha}(\hat z)u(\hat z)+f_{\alpha}(\hat z)\big\}\geq 0,
\end{equation*}
where
\begin{equation*}
l_{\hat x,\alpha}=I^{1,\delta}[\hat x,D_x\psi(\hat x,\hat y,\hat z),\psi(\cdot,\hat y,\hat z)]+I^{2,\delta}[\hat x,D_x\psi(\hat x,\hat y,\hat z),u(\cdot)],
\end{equation*}
\begin{equation*}
l_{\hat y,\alpha}=I^{1,\delta}[\hat y,D_y\psi(\hat x,\hat y,\hat z),\psi(\hat x,\cdot,\hat z)]+I^{2,\delta}[\hat y,D_y\psi(\hat x,\hat y,\hat z),u(\cdot)],
\end{equation*}
\begin{equation*}
l_{\hat z,\alpha}=I^{1,\delta}[\hat z,-\frac{D_z\psi(\hat x,\hat y,\hat z)}{2},-\frac{\psi(\hat x,\hat y,\cdot)}{2}]+I^{2,\delta}[\hat z,-\frac{D_z\psi(\hat x,\hat y,\hat z)}{2},u(\cdot)].
\end{equation*}
Thus, for any $\epsilon>0$, there exists $\alpha_\epsilon\in\mathcal{A}$ such that
\begin{eqnarray}\label{eq::4.9}
c_{\alpha_\epsilon}(\hat x)u(\hat x)+c_{\alpha_\epsilon}(\hat y)u(\hat y)-2c_{\alpha_\epsilon}(\hat z)u(\hat z)\leq L_{\alpha_\epsilon}+N_{\alpha_\epsilon}+\epsilon,
\end{eqnarray}
where
\begin{eqnarray*}
L_{\alpha_\epsilon}&=&Tr\Big(\sigma_{\alpha_\epsilon}(\hat x)\sigma_{\alpha_\epsilon}^T(\hat x)X+\sigma_{\alpha_\epsilon}(\hat y)\sigma_{\alpha_\epsilon}^T(\hat y)Y-\sigma_{\alpha_\epsilon}(\hat z)\sigma_{\alpha_\epsilon}^T(\hat z)Z\Big)\\
&&-\Big(b_{\alpha_\epsilon}(\hat x)\cdot D_x\psi(\hat x,\hat y,\hat z)+b_{\alpha_\epsilon}(\hat y)\cdot D_y\psi(\hat x,\hat y,\hat z)+b_{\alpha_\epsilon}(\hat z)\cdot D_z\psi(\hat x,\hat y,\hat z)\Big)\\
&&+2f_{\alpha_\epsilon}(\hat z)-f_{\alpha_\epsilon}(\hat y)-f_{\alpha_\epsilon}(\hat x)
\end{eqnarray*}
and
\begin{equation*}
N_{\alpha_\epsilon}=l_{\hat x,\alpha_\epsilon}+l_{\hat y,\alpha_\epsilon}-2l_{\hat z,\alpha_\epsilon}.
\end{equation*}
Since $c_\alpha\in C^{1,\bar{\theta}-1}(\bar{\mathbb R}^n)$ uniformly in $\alpha\in\mathcal{A}$ and $u\in C^{0,1}(\bar{\mathbb R}^n)$, using Lemma \ref{lemm4.2}, we have
\begin{eqnarray}\label{eq::::4.13}
&&c_{\alpha_\epsilon}(\hat x)u(\hat x)+c_{\alpha_\epsilon}(\hat y)u(\hat y)-2c_{\alpha_\epsilon}(\hat z)u(\hat z)\nonumber\\
&=&c_{\alpha_\epsilon}(\hat z)\big(u(\hat x)+u(\hat y)-2u(\hat z)\big)+\big(c_{\alpha_\epsilon}(\hat x)+c_{\alpha_\epsilon}(\hat y)-2c_{\alpha_\epsilon}(\hat z)\big)u(\hat z)\nonumber\\
&&+\big(c_{\alpha_\epsilon}(\hat x)-c_{\alpha_\epsilon}(\hat z)\big)\big(u(\hat x)-u(\hat z)\big)+\big(c_{\alpha_\epsilon}(\hat y)-c_{\alpha_\epsilon}(\hat z)\big)\big(u(\hat y)-u(\hat z)\big)\nonumber\\
&\geq&\gamma\big(u(\hat x)+u(\hat y)-2u(\hat z)\big)-|u|_{0;\mathbb R^n}\Big(\frac{\sqrt{n}}{2}\sup_{\alpha\in\mathcal{A}}[c_\alpha]_{1,\bar\theta-1;\mathbb R^n}|\hat x-\hat y|^{\bar\theta}+\sup_{\alpha\in\mathcal{A}}[c_\alpha]_{0,1;\mathbb R^n}|\hat x+\hat y-2\hat z|\Big)\nonumber\\
&&-8\max\{|u|_{0;\mathbb R^n},[u]_{0,1;\mathbb R^n}\}\sup_{\alpha\in\mathcal{A}}\max\{|c_\alpha|_{0;\mathbb R^n},[c_\alpha]_{0,1;\mathbb R^n}\}\phi(\hat x,\hat y,\hat z).
\end{eqnarray}
By ($\ref{eq::3.8}$), ($\ref{eq:4.4}$) and ($\ref{eq::4.8}$), we see that
\begin{eqnarray*}
&&Tr\Big(\sigma_{\alpha_\epsilon}(\hat x)\sigma_{\alpha_\epsilon}^T(\hat x)X+\sigma_{\alpha_\epsilon}(\hat y)\sigma_{\alpha_\epsilon}^T(\hat y)Y-\sigma_{\alpha_\epsilon}(\hat z)\sigma_{\alpha_\epsilon}^T(\hat z)Z\Big)\\
&\leq&\frac{(1+\epsilon_0)L}{\phi(\hat x,\hat y,\hat z)}\Big\{\bar\theta(2\bar\theta-1)|\hat x-\hat y|^{2\bar\theta-2}Tr\Big[\big(\sigma_{\alpha_\epsilon}(\hat x)-\sigma_{\alpha_\epsilon}(\hat y)\big)\big(\sigma_{\alpha_\epsilon}(\hat x)-\sigma_{\alpha_\epsilon}(\hat y)\big)^T\Big]\\
&&+Tr\Big[\big(\sigma_{\alpha_\epsilon}(\hat x)+\sigma_{\alpha_\epsilon}(\hat y)-2\sigma_{\alpha_\epsilon}(\hat z)\big)\big(\sigma_{\alpha_\epsilon}(\hat x)+\sigma_{\alpha_\epsilon}(\hat y)-2\sigma_{\alpha_\epsilon}(\hat z)\big)^T\Big]\Big\}\\
&&+2\eta\big(C+\sup_{\alpha\in\mathcal{A}}[\sigma_\alpha]_{0,1;\mathbb R^n}|\hat x|\big)^2.
\end{eqnarray*}
Thus, we can estimate the local term $L_{\alpha_\epsilon}$ easily. By ($\ref{eq::3.8}$), ($\ref{eq4.5}$), ($\ref{eq::4.8}$), uniform $\bar\theta$-semiconvexity of $f_{\alpha}$ with constant $C_5$ and Lemma \ref{lemm4.2}, we have
\begin{eqnarray}\label{eq::::4.14}
L_{\alpha_\epsilon}&\leq&\frac{(1+\epsilon_0)L}{\phi(\hat x,\hat y,\hat z)}\Big\{\bar\theta(2\bar\theta-1)|\hat x-\hat y|^{2\bar\theta-2}Tr\Big[\big(\sigma_{\alpha_\epsilon}(\hat x)-\sigma_{\alpha_\epsilon}(\hat y)\big)\big(\sigma_{\alpha_\epsilon}(\hat x)-\sigma_{\alpha_\epsilon}(\hat y)\big)^T\Big]\nonumber\\
&&+Tr\Big[\big(\sigma_{\alpha_\epsilon}(\hat x)+\sigma_{\alpha_\epsilon}(\hat y)-2\sigma_{\alpha_\epsilon}(\hat z)\big)\big(\sigma_{\alpha_\epsilon}(\hat x)+\sigma_{\alpha_\epsilon}(\hat y)-2\sigma_{\alpha_\epsilon}(\hat z)\big)^T\Big]\Big\}\nonumber\\
&&+2\eta\big(C+\sup_{\alpha\in\mathcal{A}}[\sigma_\alpha]_{0,1;\mathbb R^n}|\hat x|\big)^2+\frac{\bar\theta L|\hat x-\hat y|^{2\bar\theta-2}}{\phi(\hat x,\hat y,\hat z)}(\hat x-\hat y)\cdot\big(b_{\alpha_\epsilon}(\hat y)-b_{\alpha_\epsilon}(\hat x)\big)\nonumber\\
&&+\frac{L}{\phi(\hat x,\hat y,\hat z)}(\hat x+\hat y-2\hat z)\cdot\big(2b_{\alpha_\epsilon}(\hat z)-b_{\alpha_\epsilon}(\hat x)-b_{\alpha_\epsilon}(\hat y)\big)+2\eta\big(C|\hat x|+\sup_{\alpha\in\mathcal{A}}[b_\alpha]_{0,1;\mathbb R^n}|\hat x|^2\big)\nonumber\\
&&+C_5|\hat x-\hat y|^{\bar\theta}+\sup_{\alpha\in\mathcal{A}}[f_\alpha]_{0,1;\mathbb R^n}|\hat x+\hat y-2\hat z|.
\end{eqnarray}
Similarly as in the proof of Theorem \ref{th2.1}, we have $\eta|\hat x|^2\to 0$ as $\eta\to 0$ and
\begin{equation*}
\epsilon_1\leq \phi(\hat x,\hat y,\hat z)\leq \epsilon_1^{-1},
\end{equation*}
where $\epsilon_1$ is a positive constant independent of $\eta$. Letting $\delta\to 0$, $\eta\to 0$, $\epsilon\to 0$ and $\epsilon_0\to 0$ in ($\ref{eq::4.9}$), we have, by ($\ref{eq4.2}$), ($\ref{eq::::4.13}$), ($\ref{eq::::4.14}$) and the same estimates on the nonlocal term $N_{\alpha_\epsilon}$ as Theorem $\ref{th3.1}$
\begin{eqnarray*}
&&\gamma L\phi(\hat x,\hat y,\hat z)\\
&\leq&L\sup_{\alpha\in\mathcal{A}}\phi(\hat x,\hat y,\hat z)^{-1}\Big\{\bar\theta(2\bar\theta-1)|\hat x-\hat y|^{2\bar\theta-2}Tr\Big[\big(\sigma_{\alpha}(\hat x)-\sigma_{\alpha}(\hat y)\big)\big(\sigma_{\alpha}(\hat x)-\sigma_{\alpha}(\hat y)\big)^T\Big]\nonumber\\
&&+Tr\Big[\big(\sigma_{\alpha}(\hat x)+\sigma_{\alpha}(\hat y)-2\sigma_{\alpha}(\hat z)\big)\big(\sigma_{\alpha}(\hat x)+\sigma_{\alpha}(\hat y)-2\sigma_{\alpha}(\hat z)\big)^T\Big]\\
&&+\bar\theta|\hat x-\hat y|^{2\bar\theta-2}(\hat x-\hat y)\cdot\big(b_{\alpha}(\hat y)-b_{\alpha}(\hat x)\big)+(\hat x+\hat y-2\hat z)\cdot\big(2b_{\alpha}(\hat z)-b_{\alpha}(\hat x)-b_{\alpha}(\hat y)\big)\\
&&+\phi(\hat x,\hat y,\hat z)\int_{\mathbb R^n}\Big[\phi\big(\hat x+j_\alpha(\hat x,\xi),\hat y+j_\alpha(\hat y,\xi),\hat z+j_\alpha(\hat z,\xi)\big)-\phi(\hat x,\hat y,\hat z)\\
&&-\mathbbm{1}_{B_1(0)}(\xi)\Big(D_x\phi(\hat x,\hat y,\hat z),D_y\phi(\hat x,\hat y,\hat z),D_z\phi(\hat x,\hat y,\hat z)\Big)\cdot \Big(j_\alpha(\hat x,\xi),j_\alpha(\hat y,\xi),j_\alpha(\hat z,\xi)\Big)\Big]\mu(d\xi)\Big\}\\
&&+C_5|\hat x-\hat y|^{\bar\theta}+\sup_{\alpha\in\mathcal{A}}[f_\alpha]_{0,1;\mathbb R^n}|\hat x+\hat y-2\hat z|\\
&&+|u|_{0;\mathbb R^n}\Big(\frac{\sqrt{n}}{2}\sup_{\alpha\in\mathcal{A}}[c_\alpha]_{1,\bar\theta-1;\mathbb R^n}|\hat x-\hat y|^{\bar\theta}+\sup_{\alpha\in\mathcal{A}}[c_\alpha]_{0,1;\mathbb R^n}|\hat x+\hat y-2\hat z|\Big)\nonumber\\
&&+8\max\{|u|_{0;\mathbb R^n},[u]_{0,1;\mathbb R^n}\}\sup_{\alpha\in\mathcal{A}}\max\{|c_\alpha|_{0;\mathbb R^n},[c_\alpha]_{0,1;\mathbb R^n}\}\phi(\hat x,\hat y,\hat z).
\end{eqnarray*}
Therefore,
\begin{eqnarray}
\gamma&\leq&\sup_{\alpha\in\mathcal{A}}\phi(\hat x,\hat y,\hat z)^{-2}\Big\{\bar\theta(2\bar\theta-1)|\hat x-\hat y|^{2\bar\theta-2}Tr\Big[\big(\sigma_{\alpha}(\hat x)-\sigma_{\alpha}(\hat y)\big)\big(\sigma_{\alpha}(\hat x)-\sigma_{\alpha}(\hat y)\big)^T\Big]\nonumber\\
&&+Tr\Big[\big(\sigma_{\alpha}(\hat x)+\sigma_{\alpha}(\hat y)-2\sigma_{\alpha}(\hat z)\big)\big(\sigma_{\alpha}(\hat x)+\sigma_{\alpha}(\hat y)-2\sigma_{\alpha}(\hat z)\big)^T\Big]\nonumber\\
&&+\bar\theta|\hat x-\hat y|^{2\bar\theta-2}(\hat x-\hat y)\cdot\big(b_{\alpha}(\hat y)-b_{\alpha}(\hat x)\big)+(\hat x+\hat y-2\hat z)\cdot\big(2b_{\alpha}(\hat z)-b_{\alpha}(\hat x)-b_{\alpha}(\hat y)\big)\nonumber\\
&&+\phi(\hat x,\hat y,\hat z)\int_{\mathbb R^n}\Big[\phi\big(\hat x+j_\alpha(\hat x,\xi),\hat y+j_\alpha(\hat y,\xi),\hat z+j_\alpha(\hat z,\xi)\big)-\phi(\hat x,\hat y,\hat z)\nonumber\\
&&-\mathbbm{1}_{B_1(0)}(\xi)\Big(D_x\phi(\hat x,\hat y,\hat z),D_y\phi(\hat x,\hat y,\hat z),D_z\phi(\hat x,\hat y,\hat z)\Big)\cdot \Big(j_\alpha(\hat x,\xi),j_\alpha(\hat y,\xi),j_\alpha(\hat z,\xi)\Big)\Big]\mu(d\xi)\Big\}\nonumber\\
&&+\frac{C_6}{L}\leq N_2+\frac{C_6}{L},\label{eq,,4.19}
\end{eqnarray}
where $N_2$ is defined in ($\ref{eq,4.13}$) and $C_6$ is a positive constant. Hence, if $\gamma>N_2$, we can find a sufficiently large $L$ such that we have a contradiction in ($\ref{eq,,4.19}$). Therefore, $u$ is $\bar\theta$-semiconcave in $\mathbb R^n$.
\end{proof}

\textbf{Acknowledgement.} I would like to thank my advisor Prof. Andrzej \Swiech\, for suggesting the problem and for all the useful discussions and encouragement.

\end{document}